\documentclass[12pt]{article}
\usepackage{fullpage,amsthm,amsmath,amssymb,xspace,verbatim,tikz,subfig,setspace,url}
\usetikzlibrary{positioning, shapes.geometric, calc}
\tikzstyle{vertex}=[circle,fill=black,inner sep=2pt]

\newcounter{thmctr}
\newtheorem{thm}[thmctr]{Theorem}
\newtheorem{lemma}[thmctr]{Lemma}
\newtheorem{prop}[thmctr]{Proposition}

\newtheorem*{definition}{Definition}

\theoremstyle{definition}

\newtheorem{problem}[thmctr]{Problem}
\theoremstyle{plain}

\newcommand{\propdiscp}{\texttt{Disc}$_p$\xspace}

\newcommand{\propexpandp}[1]{\texttt{Expand}$_p$[#1]\xspace}
\newcommand{\propcountp}[1]{\texttt{Count}$_p$[#1]\xspace}
\newcommand{\propcyclep}[2][4]{\texttt{Cycle$_{p,#1}$}[#2]\xspace}

\newcommand{\propcdp}[1]{\texttt{CliqueDisc}$_p$[#1]\xspace}
\newcommand{\proppartexpp}[1]{\texttt{PartiteExpand}$_p$[#1]\xspace}

\newcommand{\propexpandh}[1]{\texttt{Expand}$_{1/2}$[#1]\xspace}
\newcommand{\propcdh}[1]{\texttt{CliqueDisc}$_{1/2}$[#1]\xspace}

\newcommand{\propexpand}[1]{\texttt{Expand}[#1]\xspace}
\newcommand{\propcount}[1]{\texttt{Count}[#1]\xspace}
\newcommand{\propcycle}[2][4]{\texttt{Cycle$_{#1}$}[#2]\xspace}

\newcommand{\propcd}[1]{\texttt{CliqueDisc}[#1]\xspace}
\newcommand{\proppartexp}[1]{\texttt{PartiteExpand}[#1]\xspace}
\newcommand{\propdev}[1]{\texttt{Deviation}[#1]\xspace}
\DeclareMathOperator{\dev}{\textnormal{dev}}

\newcommand{\dhruvuni}{University of Illinois at Chicago \\ mubayi@uic.edu}
\newcommand{\johnuni}{University of Illinois at Chicago \\ lenz@math.uic.edu}
\newcommand{\dhruvfoot}{\footnote{Research supported in part by  NSF Grants 0969092 and 1300138.}}
\newcommand{\johnfoot}{\footnote{Research partly supported by NSA Grant H98230-13-1-0224.}}

\title{The Poset of Hypergraph Quasirandomness}
\author{John Lenz \johnfoot \\ \johnuni \and Dhruv Mubayi \dhruvfoot \\ \dhruvuni}

\begin{document}
\maketitle

\begin{abstract}
Chung and Graham began the systematic study of $k$-uniform hypergraph quasirandom properties soon
after the foundational results of Thomason and Chung-Graham-Wilson on quasirandom graphs. One
feature that became apparent in the early work on $k$-uniform hypergraph quasirandomness is that
properties that are equivalent for graphs are not equivalent for hypergraphs, and thus hypergraphs
enjoy a variety of inequivalent quasirandom properties. In the past two decades, there has been an
intensive study of these disparate notions of quasirandomness for hypergraphs, and an
open problem that has emerged is to determine the relationship between them.

Our main result is to determine the poset of implications between these quasirandom properties. This
answers a recent question of Chung and continues a project begun by Chung and Graham in their
first paper on hypergraph quasirandomness in the early 1990's. 
\end{abstract}

\section{Introduction}

An important line of research in extremal combinatorics and computer science in the last few decades
is the study of quasirandom or pseudorandom structures.  This was initiated by
Thomason~\cite{qsi-thomason87, qsi-thomason87-2} and Chung, Graham, and Wilson~\cite{qsi-chung89},
who studied explicitly constructed graphs which mimic the random graph. Applications of
quasirandom structures have appeared in many situations in extremal combinatorics and computer
science, for example in recent proofs of Szemer\'edi's Theorem~\cite{sze-szemeredi75} using the
Strong Hypergraph Regularity Lemma~\cite{rrl-gowers07,rrl-nagle06,rrl-rodl04,rrl-rodl06,rrl-tao06}
and in expander graphs~\cite{ee-survey06} in computer science. For details on quasirandomness, we
refer the reader to a survey of Krivelevich and Sudakov~\cite{qsi-survey-krivelevich06} for graph
quasirandomness and recent papers of Gowers~\cite{hqsi-gowers06,rrl-gowers07,qsi-gowers08} for
other quasirandom structures.

Soon after the papers~\cite{qsi-thomason87,qsi-thomason87-2} and~\cite{qsi-chung89}, Chung and
Graham~\cite{hqsi-chung90-2} initiated the study of quasirandomness in hypergraphs.  Since these
early papers on the subject, there have been a variety of different notions of quasirandomness
defined for hypergraphs, and the relationships between these quasirandom properties are not
completely understood. Chung~\cite{hqsi-chung12, hqsi-chung90} posed the following problem.

\begin{problem} \label{prob:relationship}
  \textbf{(Chung~\cite{hqsi-chung12, hqsi-chung90})}
  How is a given property placed in the quasirandom hierarchy and what is the lattice structure
  illustrating the relationship among quasirandom properties of hypergraphs? 
\end{problem}

Our main result is to answer this question for many $k$-uniform hypergraph quasirandom properties.

A \emph{$k$-uniform hypergraph} is a pair of finite sets $(V(G),E(G))$ such that $E(G)$ is a
collection of $k$-subsets of $V(G)$. The set $V(G)$ is the vertex set and $E(G)$ is the edge set.
For a hypergraph $G$ and $U \subseteq V(G)$, the induced subhypergraph on $U$, denoted $G[U]$, is the
hypergraph with vertex set $U$ and edge set $\{ e \in E(G) : e \subseteq U \}$.  A \emph{graph} is a
$2$-uniform hypergraph. Let $\mathcal{G} = \{G_n\}_{n\rightarrow\infty}$ be a sequence of graphs
with $|V(G_n)| = n$ and let $0 < p < 1$ be a fixed real.  The graph sequence $\mathcal{G}$ is
\emph{$p$-quasirandom} if it satisfies the following properties.
\begin{itemize}
    \item \texttt{Disc}$_p$:  (short for discrepency) for every $U \subseteq V(G_n)$, $|E(G_n[U])| = p
      \binom{|U|}{2} + o(n^2)$.

    \item \texttt{Expand}$_p$:  For every $S,T \subseteq V(G_n)$, $e(S,T) = p|S||T| + o(n^2)$, where
      $e(S,T)$ is the number of edges with one endpoint in $S$ and one endpoint in $T$, with edges
      inside $S \cap T$ counted twice.
\end{itemize}

The use of little-$o$ notation in the above definitions requires some explanation.  The precise
definition of \texttt{Disc}$_p$ is the property of graph sequences defined as follows: $\mathcal{G}
= \{G_n\}_{n\rightarrow\infty}$ with $|V(G_n)| = n$ satisfies \texttt{Disc}$_p$ if there exists a
function $f : \mathbb{N} \rightarrow \mathbb{R}$ such that $f(n) = o(n^2)$ (i.e.\
$\lim_{n\rightarrow\infty} f(n)n^{-2} = 0$) so that for all $n$ and all $U \subseteq V(G_n)$,
$p\binom{|U|}{2} - f(n) \leq |E(G_n[U])| \leq p \binom{|U|}{2} + f(n)$.  \texttt{Expand}$_p$ is
defined similarly. 

It is easy to see that \texttt{Disc}$_p$ and \texttt{Expand}$_p$ are equivalent; \texttt{Expand}$_p$
$\Rightarrow$ \propdiscp is trivial by letting $S = T = U$ and the converse is a simple
inclusion/exclusion argument.  In addition, \texttt{Disc}$_p$ and \texttt{Expand}$_p$ are both
central properties of the random graph.  Many more properties of graph sequences have been shown
equivalent to \propdiscp and \texttt{Expand}$_p$, including counting subgraphs, counting induced
subgraphs, spectral conditions, sizes of common neighborhoods, and counting even/odd subgraphs of
cycles, see~\cite{qsi-chung89, qsi-dense-jason11, qsi-dense-lovasz08, qsi-dense-myers02,
qsi-dense-myers05, qsi-dense-nikiforov06, qsi-dense-shapira08, qsi-dense-shapira10,
qsi-dense-simonovits91, qsi-dense-simonovits97, qsi-dense-simonovits03, qsi-dense-skokan04,
qsi-dense-yuster10}.  In addition, several researchers investigated the sparse case where $p$ is no
longer a constant but $p = o(1)$, see~\cite{qsi-sparse-alon10, qsi-sparse-chung02,
qsi-sparse-chung08, qsi-sparse-kohayakawa03, qsi-sparse-kohayakawa04}.  In this paper, we will be
concentrating only on the dense case when $p$ is a fixed constant.

For $k$-uniform hypergraphs, there are several obvious generalizations of the graph properties
\texttt{Disc}$_p$ and \texttt{Expand}$_p$ which we discuss next.  A \emph{proper partition} $\pi$ of
$k$ is an unordered list of at least two positive integers whose sum is $k$.  For the partition
$\pi$ of $k$ given by $k = k_1 + \dots + k_t$, we will abuse notation by saying that $\pi = k_1 +
\dots + k_t$.  Let $\mathcal{H} = \{H_n\}_{n\rightarrow\infty}$ be a sequence of $k$-uniform
hypergraphs with $|V(H_n)| = n$ and let $0 < p < 1$ be a fixed integer.  For a proper partition $\pi
= k_1 + \dots + k_t$ of $k$, define the following properties of $\mathcal{H}$.
\begin{itemize}
  \item \propdiscp:  for every $U \subseteq V(H_n)$, $|E(H_n[U])| = p \binom{|U|}{k} +
    o(n^k)$.

  \item \propexpandp{$\pi$}: For all $S_i \subseteq \binom{V(H_n)}{k_i}$ where $1 \leq i \leq t$,
    \begin{align*} 
      e(S_1,\dots,S_t) =  p \prod_{i=1}^t \left| S_i \right| + o(n^{k})
    \end{align*}
    where $e(S_1,\dots,S_t)$ is the number of tuples $(s_1,\dots,s_t)$ such that $s_1 \cup
    \dots \cup s_t$ is a hyperedge and $s_i \in S_i$.
\end{itemize}
\propexpandp{$1+\dots+1$} $\Rightarrow$ \propdiscp is easy by letting $S_i = U$ and an
inclusion/exclusion argument shows \propdiscp $\Rightarrow$ \propexpandp{$1+\dots+1$} (see
Lemma~\ref{lem:cdtoexp} for a proof of a more general statement). One of the most important graph
properties equivalent to \propdiscp is \propcountp{All}, the property that for all graphs $F$, the
number of labeled copies of $F$ in $G_n$ is $p^{|E(F)|} n^{|V(F)|} + o(n^{|V(F)|})$ and at first
glance one might suspect this equivalence also holds for hypergraphs.  However, R\"odl observed that
a three-uniform construction of Erd\H{o}s and Hajnal~\cite{ram-erdos72} satisfies
\texttt{Disc}$_{1/4}$ and fails \texttt{Count}$_{1/4}$[All].  In light of this construction, Frankl and R\"odl
suggested the following property which can be seen as an alternate generalization of \propdiscp from
graphs to hypergraphs.  Let $\mathcal{H} = \{H_n\}_{n\rightarrow\infty}$ be a sequence of
$k$-uniform hypergraphs with $|V(H_n)| = n$, let $0 < p < 1$ be a fixed integer, and let $1 \leq
\ell \leq k -1$ be an integer and define the following property.

\begin{itemize}
  \item \propcdp{$\ell$}: for every $\ell$-uniform hypergraph $G$ where $V(G) = V(H_n)$, $|E(H_n) \cap
    \mathcal{K}_k(G)| = p |\mathcal{K}_k(G)| + o(n^k)$, where $\mathcal{K}_k(G)$ is set of
    \emph{$k$-cliques of $G$}, the collection of $k$-sets $T \subseteq V(G)$ such that all
    $\ell$-subsets of $T$ are edges of $G$.
\end{itemize}

Note that for $k$-uniform hypergraphs and $\ell = 1$, \propcdp{$1$} $\Leftrightarrow$ \propdiscp by
definition so \propcdp{$\ell$} is a generalization of \propdiscp.  Many hypergraph quasirandom
properties are equivalent to \propcdp{$\ell$} and \propexpandp{$\pi$} for some $\ell$ or $\pi$.  See
\cite{hqsi-austin10, hqsi-chung12, hqsi-chung90, hqsi-chung90-2, hqsi-chung91, hqsi-chung92,
hqsi-conlon12, hqsi-gowers06, hqsi-keevash09, hqsi-kohayakawa10, hqsi-kohayakawa02,
hqsi-lenz-quasi12} for the studies of these properties, which include counting subhypergraphs,
counting induced subhypergraphs, spectral characterizations, and counting even/odd subgraphs.

There are two more hypergraph quasirandom properties that have been studied.  First, Chung and
Graham's~\cite{hqsi-chung90-2} original property on even/odd subgraphs of the octahedron called
\propdev{$\ell$} and an extension of \propcdp{$\ell$} recently proposed by Chung~\cite{hqsi-chung12}.

\begin{itemize}
  \item For $2 \leq \ell \leq k$, define \propdev{$\ell$} as follows:
    \begin{align*}
      \sum_{\substack{x_1, \dots, x_{k-\ell} \in V(H) \\
            y_{1,0}, y_{1,1}, \dots, y_{\ell,0}, y_{\ell,1} \in V(H) }}
            (-1)^{|\mathcal{O}[\vec{x},\vec{y}] \cap E(H)|} = o(n^{k+\ell}),
    \end{align*}
    where $\mathcal{O}[\vec{x},\vec{y}]$ is the collection of hyperedges of the squashed octahedron.
    That is, $\mathcal{O}[\vec{x},\vec{y}] = \{ \{
    x_1,\dots,x_{k-\ell},y_{1,i_1},\dots,y_{\ell,i_\ell}\} : 0 \leq i_j \leq 1 \}$.  Conceptually,
    \propdev{$\ell$} states that the difference between the number of even and odd squashed
    octahedrons is negligible compared to the number of squashed octahedrons.

  \item For $1 \leq \ell \leq k - 1$ and $1 \leq s \leq \binom{k}{\ell}$, define \propcdp{$\ell,s$}
    as follows: for every $\ell$-uniform hypergraph $G$ where $V(G) \subseteq
    V(H_n)$,\footnote{This is slightly different than Chung's~\cite{hqsi-chung12} definition; she
    defined \propcd{$\ell,s$} only for spanning $G$. We believe the two definitions are equivalent and
    have proved this for several small cases.}
    \begin{align*}
      \left| \{ T \in E(H_n) : |E(G[T])| \geq s \} \right|
      = p \left| \left\{ T \in \binom{V(G)}{k} : |E(G[T])| \geq s \right\} \right|
      + o(n^k).
    \end{align*}
\end{itemize}

Although it is possible to extend the definition of \propdev{$\ell$} to arbitrary $0 < p < 1$, the
deviation property has been studied primarily for $p = \frac{1}{2}$, which is how we have stated it.
Also, note that \propcdp{$\ell,\binom{k}{\ell}$} is the same property as \propcdp{$\ell$}.  

\begin{figure}[ht] 
\begin{center}
\begin{tikzpicture}[yscale=1.1]
  \node (dev6) at (1,2.8) {Dev(6)};
  \node (cd5) [left=0cm of dev6.west] {CD(5) $\Leftrightarrow$};
  \node (dev5) at (1,2.0) {Dev(5)};
  \node (dev4) at (1,1.2) {Dev(4)};
  \node (dev3) at (1,0.4) {Dev(3)};
  \node (dev2) at (1,-0.4) {Dev(2)};

  \node (cd4) at (-3.2,0.5) {CD($4$)};
  \node (cd3) at (-3.2,-0.5) {CD($3$)};
  \node (cd2) at (-3.2,-1.6) {CD($2$)};

  \node (p33) at (-1,-2) {$(3,3)$};
  \node (p42) at (1,-2) {$(4,2)$};
  \node (p51) at (3,-2) {$(5,1)$};

  \node (p222) at (-1,-3) {$(2,2,2)$};
  \node (p321) at (1,-3) {$(3,2,1)$};
  \node (p411) at (3,-3) {$(4,1,1)$};

  \node (p2211) at (0,-4) {$(2,2,1,1)$};
  \node (p3111) at (2,-4) {$(3,1,1,1)$};

  \node (p21111) at (1,-5) {$(2,1,1,1,1)$};

  \node (p111111) at (1,-6) {$(1,1,1,1,1,1)$};
  \node (bottom) [right=0cm of p111111.east] {$\Leftrightarrow$ Disc};
  \node (cd1) [left=0cm of p111111.west] {CD(1) $\Leftrightarrow$};

  \draw (dev2) -- (dev3) -- (dev4) -- (dev5) -- (dev6);
  \draw (cd2) -- (cd3) -- (cd4);
  \draw (p111111) -- (p21111) -- (p2211) -- (p222) -- (p42) -- (dev2);
  \draw (p21111) -- (p3111) -- (p411) -- (p51) -- (dev2);
  \draw (p2211) -- (p321) -- (p33) -- (dev2);
  \draw (p3111) -- (p321) -- (p42);
  \draw (p222) -- (p42);
  \draw (p321) -- (p51);
  \draw (p411) -- (p42);
  \draw (cd4) -- (p42);
  \draw (cd3) -- (p33);
  \draw (cd2) -- (p222);
  \draw (p2211) -- (p411);
  \draw (dev5) to (cd4);
  \draw (dev4) to (cd3);
  \draw (dev3) to (cd2);
\end{tikzpicture}
\caption{The Hasse diagram of quasirandom properties for $k = 6$.}
\label{fig:poset}
\end{center}
\end{figure}
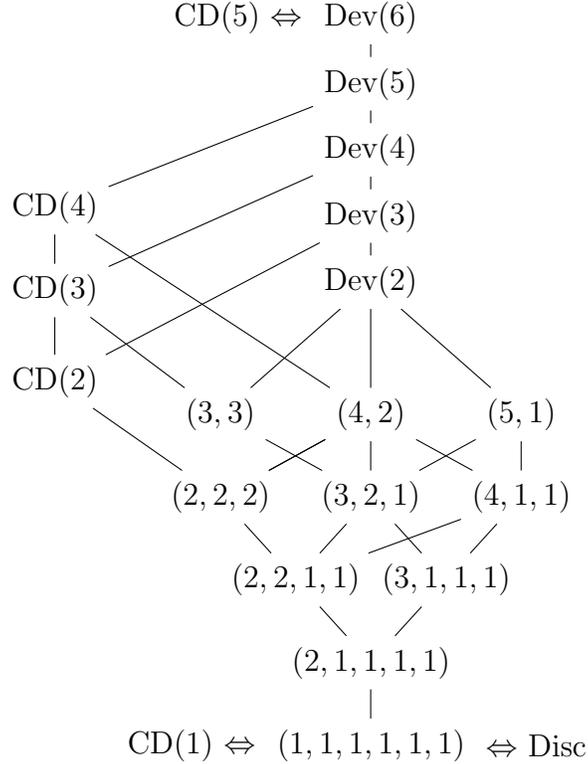 

In~\cite{hqsi-chung12, hqsi-chung90}, Chung made partial progress on
Problem~\ref{prob:relationship}; see Table~\ref{tab:main} for the exact results she proved.  Our
main result is to determine all relationships between \propexpandp{$\pi$},
\propcdp{$\ell,s$}, and \propdev{$\ell$} for all $\ell$, $s$, and $\pi$.  As a consequence, our work
also determines the relationships between other properties like counting and spectral conditions
studied in the literature, since these have been shown equivalent to one of \propexpandp{$\pi$},
\propcdp{$\ell,s$}, or \propdev{$\ell$}.  While we use some basic ideas introduced by Chung, most of
our results require new constructions for the non-implications of quasirandom properties.  The
proofs of our two main positive results (Theorems~\ref{thm:cdellsequiv} and~\ref{thm:maincdtoexp})
also use new techniques.

Our first result is that \propcdp{$\ell,s$} is a superfluous property in the sense that
\propcdp{$\ell,s$} $\Leftrightarrow$ \propcdp{$\ell,s'$} for all $\ell,s,s'$.  Since
\propcdp{$\ell$} is equivalent to \propcdp{$\ell,\binom{k}{\ell}$}, we can reduce to studying just
\propcdp{$\ell$}.

\begin{thm} \label{thm:cdellsequiv}
  Fix $k \geq 3$ and $2 \leq \ell < k$.  Then
  \propcdp{$\ell,1$} $\Leftrightarrow$ \propcdp{$\ell,2$} $\Leftrightarrow \cdots \Leftrightarrow$
  \propcdp{$\ell,\binom{k}{\ell}$}.
\end{thm}

The proof of Theorem~\ref{thm:cdellsequiv} appears in Section~\ref{sec:cdells}.  Our next result is
that the expansion properties are arranged in a poset via partition refinement.  In particular,
\propexpandp{$\pi$} is a distinct property for each $\pi$.

\begin{definition}
  A partition $\pi' = m_1+\dots+m_r$ is a \emph{refinement} of a partition $\pi = k_1+\dots+k_t$ if
  there is a surjection $\phi : \{1,\dots,r\} \rightarrow \{1,\dots,t\}$ such that for every $1 \leq
  i \leq t$, $k_i = \sum_{j : \phi(j) = i} m_j$.  If $\pi'$ is a refinement of $\pi$, we write $\pi'
  \leq \pi$.  Also, for $\pi = k_1 + \dots + k_t$, let $\max \pi = \max_i k_i$.
\end{definition}

\begin{thm} \label{thm:mainexpand}
  \propexpandp{$\pi$} $\Rightarrow$ \propexpandp{$\pi'$} if and only if $\pi'$ is a refinement of
  $\pi$.
\end{thm}

Having determined the poset of implications for the properties \propexpandp{$\pi$}, we now give the
relationships between \propcdp{$\ell$} and \propexpandp{$\pi$} for all $\ell$ and $\pi$.

\begin{thm} \label{thm:maincdtoexp}
  Let $1 \leq \ell \leq k -1$ and $\pi = k_1 + \dots + k_t$.  \propcdp{$\ell$} $\Rightarrow$
  \propexpandp{$\pi$} if and only if $k_i \leq \ell$ for all $i$.  Also, \propexpandp{$\pi$}
  $\Rightarrow$ \propcdp{$1$} for all $\pi$ but \propexpandp{$\pi$} $\not\Rightarrow$
  \propcdp{$\ell$} for any $\pi$ and $\ell \geq 2$.
\end{thm}

\begin{table}[t] 
  \centering
  \begin{tabular}{|l|l|l|}
    \hline
    Range & Result & Proof
    \\ \hline

    $\pi' \leq \pi$ &
    \propexpandp{$\pi$} $\Rightarrow$ \propexpandp{$\pi'$} &
    Theorem~\ref{thm:mainexpand}, Lemma~\ref{lem:exptoexp}
    \\ \hline

    $\pi' \not\leq \pi$ &
    \propexpandp{$\pi$} $\not\Rightarrow$ \propexpandp{$\pi'$} &
    Theorem~\ref{thm:mainexpand}, Lemma~\ref{lem:expnotexp}
    \\ \hline

    all $\pi$ &
    \propexpandp{$\pi$} $\Rightarrow$ \propcdp{$1$} &
    Theorem~\ref{thm:maincdtoexp}
    \\ \hline

    all $\pi$, $\ell \geq 2$ &
    \propexpandp{$\pi$} $\not\Rightarrow$ \propcdp{$\ell$} &
    Theorem~\ref{thm:maincdtoexp}, Lemma~\ref{lem:expnotcd}
    \\ \hline

    all $\pi$, $\ell$ &
    \propexpandh{$\pi$} $\not\Rightarrow$ \propdev{$\ell$} &
    Theorem~\ref{thm:mainwithdev}, Lemma~\ref{lem:expnotdev}
    \\ \hline

    & &
    \\ \hline

    $2 \leq \ell \leq k-1$ &
    \propcdp{$\ell$} $\Rightarrow$ \propcdp{$\ell-1$} &
    Chung~\cite{hqsi-chung90}
    \\ \hline

    $1 \leq \ell \leq k-2$ &
    \propcdp{$\ell$} $\not\Rightarrow$ \propcdp{$\ell+1$} &
    Chung~\cite{hqsi-chung90} for $p = \frac{1}{2}$, Lemma~\ref{lem:cdnotcd}
    \\ \hline

    $\max \pi \leq \ell$ &
    \propcdp{$\ell$} $\Rightarrow$ \propexpandp{$\pi$} &
    Theorem~\ref{thm:maincdtoexp}, Lemma~\ref{lem:cdtoexp}
    \\ \hline

    $\max \pi > \ell$ &
    \propcdp{$\ell$} $\not\Rightarrow$ \propexpandp{$\pi$} &
    Theorem~\ref{thm:maincdtoexp}, Lemma~\ref{lem:cdnotexp}
    \\ \hline

     &
    \propcdh{$k-1$} $\Rightarrow$ \propdev{$k$} &
    Chung and Graham~\cite{hqsi-chung90-2}
    \\ \hline

    all $\ell$ &
    \propcdh{$k-2$} $\not\Rightarrow$ \propdev{$\ell$} &
    Theorem~\ref{thm:mainwithdev}, Lemma~\ref{lem:cdnotdev}
    \\ \hline

    & &
    \\ \hline

    $3 \leq \ell \leq k$ &
    \propdev{$\ell$} $\Rightarrow$ \propdev{$\ell-1$} &
    Chung~\cite{hqsi-chung90}, Lemma~\ref{lem:devtodev}
    \\ \hline

    $2 \leq \ell \leq k -1$ &
    \propdev{$\ell$} $\not\Rightarrow$ \propdev{$\ell+1$} &
    Proposition~\ref{prop:devnotdev}, Lemma~\ref{lem:devnotdev}
    \\ \hline

    $2 \leq \ell \leq k$ &
    \propdev{$\ell$} $\Rightarrow$ \propcdh{$\ell-1$} &
    Chung~\cite{hqsi-chung90}, Lemma~\ref{lem:devtocd}
    \\ \hline

    $2 \leq \ell \leq k-1$ &
    \propdev{$\ell$} $\not\Rightarrow$ \propcdh{$\ell$} &
    Chung~\cite{hqsi-chung90}, Lemma~\ref{lem:devnotcd}
    \\ \hline

    all $\pi$, $\ell$ &
    \propdev{$\ell$} $\Rightarrow$ \propexpandh{$\pi$} &
    Theorem~\ref{thm:mainwithdev}, Lemma~\ref{lem:devtoexp}
    \\ \hline
  \end{tabular}
  \caption{Relationships between quasirandom properties}
  \label{tab:main}
\end{table} 

Next, we determine the relationships between \propdev{$\ell$} and the other properties.  Since
\propdev{$\ell$} has been studied primarily for $p = \frac{1}{2}$, we only study the relationships
between \propdev{$\ell$} and \propcdh{$\ell$} and \propexpandh{$\pi$}.

\begin{thm} \label{thm:mainwithdev}
  For all $2 \leq \ell \leq k$ and all $\pi$, we have \propdev{$\ell$} $\Rightarrow$
  \propexpandh{$\pi$}.  Furthermore, \propcdh{$k-1$} $\Rightarrow$ \propdev{$k$} but no expansion and
  no other clique discrepency implies \propdev{$\ell$} for any $\ell$.
\end{thm}

Lastly, we prove that \propdev{$\ell$} $\not\Rightarrow$ \propdev{$\ell+1$}.  When combined with the
implication \propdev{$\ell$} $\Rightarrow$ \propdev{$\ell-1$} (Chung~\cite{hqsi-chung90}), this
proves that the properties \propdev{$\ell$} form a chain of distinct hypergraph quasirandom
properties.

\begin{prop} \label{prop:devnotdev}
  For all $2 \leq \ell \leq k -1$, we have \propdev{$\ell$} $\not\Rightarrow$ \propdev{$\ell+1$}.
\end{prop}

The proofs of Theorems~\ref{thm:mainexpand},~\ref{thm:maincdtoexp},~\ref{thm:mainwithdev}, and
Proposition~\ref{prop:devnotdev} appear in Sections~\ref{sec:implications},~\ref{sec:constructions},
and~\ref{sec:notimplications}.  Together with results of Chung~\cite{hqsi-chung90} and Chung and
Graham~\cite{hqsi-chung90-2}, these theorems complete the characterization between
\propexpandp{$\pi$}, \propcdp{$\ell$}, and \propdev{$\ell$} for all $\ell$ and $\pi$.
Table~\ref{tab:main} summarizes these results and states where each piece is proved.
Figure~\ref{fig:poset} shows a diagram of the relationships for $k = 6$.

The remainder of this paper is organized as follows.  In Section~\ref{sec:implications}, we prove
the implications in Table~\ref{tab:main} (Lemmas~\ref{lem:exptoexp},~\ref{lem:cdtoexp},
and~\ref{lem:devtoexp}).  In Section~\ref{sec:constructions}, we define three families of
constructions which are used to show the separation of quasirandom properties, and in
Section~\ref{sec:notimplications} we use these constructions to prove all the negative implications
in Table~\ref{tab:main}.  Section~\ref{sec:cdells} contains the proof of
Theorem~\ref{thm:cdellsequiv}.  Lastly, Appendix~\ref{sec:chungdeviation} contains for completeness
some proofs of results of Chung~\cite{hqsi-chung90} that are used in this paper.   The subscript $p$
on the quasirandom properties is dropped if it is clear from context.

\section{Implications} 
\label{sec:implications}

In this section, we prove the implications in Table~\ref{tab:main}.

\subsection{Expansion} 
\label{sub:implicationsexp}

Our goal in this subsection is to prove Lemma~\ref{lem:exptoexp} below.
First, we introduce a variant of \propexpand{$\pi$} where the sets are disjoint.  Let $\pi = k_1 +
\dots + k_t$ be a proper partition of $k$.  If $H$ is a hypergraph and $S_1,\dots,S_t$ are sets such
that $S_i \subseteq \binom{V(H)}{k_i}$, denote by $V(S_i) = \cup_{s \in S_i} s$ and call
$S_1,\dots,S_t$ \emph{disjoint} if $V(S_i) \cap V(S_j) = \emptyset$ for $i \neq j$.  Let
$\mathcal{H} = \{H_n\}_{n\rightarrow\infty}$ be a sequence of $k$-uniform hypergraphs such that
$|V(H_n)| = n$.  Define the following property of the sequence $\mathcal{H}$.

\begin{itemize}
  \item \proppartexpp{$\pi$}: For all $S_i \subseteq \binom{V(H_n)}{k_i}$ where $S_1,\dots,S_t$ are
    disjoint,
    \begin{align*} 
      e(S_1,\dots,S_t) =  p \prod_{i=1}^t \left| S_i \right| + o\left(n^{k}\right)
    \end{align*}
    where $e(S_1,\dots,S_t)$ is the number of tuples $(s_1,\dots,s_t)$ such that $s_i \in S_i$ for
    all $i$ and $s_1 \cup \dots \cup s_t \in E(H_n)$.
\end{itemize}

\begin{lemma} \label{lem:partitetonormal}
  \proppartexpp{$\pi$} $\Rightarrow$ \propexpandp{$\pi$}.
\end{lemma}

\begin{proof} 
Let $\mathcal{H} = \{H_n\}_{n \rightarrow\infty}$ be a sequence of hypergraphs satisfying
\proppartexp{$\pi$}.  Throughout this proof, for notational simplicity we drop the subscript $n$.
Let $S_i \subseteq \binom{V(H)}{k_i}$ be given.  Let $\mathcal{P} = (P_1,\dots,P_t)$ be an ordered
partition of $V(H)$ into $t$ non-empty parts.  That is, $\mathcal{P}$ is an ordered tuple of $t$
non-empty vertex sets such that $P_i \cap P_j = \emptyset$ for $i \neq j$ and $\cup P_i = V(H)$.
For $1 \leq i \leq t$, define $S_i[P_i]$ to be the collection of $k_i$-sets in $S_i$ which are
subsets of $P_i$.  Then
\begin{align*}
  e(S_1,\dots,S_t) = \frac{1}{t^{n-k}} \sum_{\mathcal{P}} e(S_1[P_1], \dots, S_t[P_t]),
\end{align*}
since in the sum over partitions, each $(s_1,\dots,s_t) \in S_1 \times \dots \times S_t$ with $s_1
\cup \dots \cup s_t \in E(H)$ is counted $t^{n-k}$ times.  That is, if $E = s_1\cup\dots\cup s_t$ is
an edge with $s_i \in S_i$, then the partitions which count $(s_1,\dots,s_t)$ are the partitions
formed by starting with $P_1 = s_1, \dots, P_t =  s_t$ and adding the other $n-k$ vertices
arbitrarily to the $t$ parts.

A similar argument shows that
\begin{align} \label{eq:partitetonormal}
  |S_1| \cdots |S_t| = \frac{1}{t^{n-k}} \sum_{\mathcal{P}} |S_1[P_1]| \cdots |S_t[P_t]|.
\end{align}
Now apply \proppartexp{$\pi$} to $S_1[P_1], \dots, S_t[P_t]$ to obtain
\begin{align*}
  e(S_1,\dots,S_t) 
  &= \frac{1}{t^{n-k}} \sum_{\mathcal{P}} \Big( p |S_1[P_1]| \cdots |S_t[P_t]| + o(n^k)  \Big) \\
  &= \frac{p}{t^{n-k}} \sum_{\mathcal{P}} |S_1[P_1]| \cdots |S_t[P_t]| + 
  \sum_{\mathcal{P}} o\left(\frac{n^k}{t^{n-k}} \right) \\
  &= p |S_1| \cdots |S_t| + o \left( \frac{n^k t! S(n,t)}{t^{n-k}} \right).
\end{align*}
The last equality combines \eqref{eq:partitetonormal} with the fact that the number of partitions in
the sum is $t! S(n,t)$ where $S(n,t)$ is the Stirling number of the second kind.  Trivially,
$\frac{t^{n-t}}{t!} \leq S(n,t) \leq t^n$ so that $S(n,t) =
\Theta(t^n)$.  Since $t$ and $k$ are fixed, $\frac{n^k t! S(n,t)}{t^{n-k}} = \Theta(n^k)$ implying that
$e(S_1,\dots,S_t) = p |S_1| \cdots |S_t| + o(n^k)$, completing the proof.
\end{proof} 

\begin{lemma} \label{lem:exptoexp}
  If $\pi'$ is a refinement of $\pi$, then \propexpandp{$\pi$} $\Rightarrow$ \propexpandp{$\pi'$}.
\end{lemma}

\begin{proof} 
Let $\mathcal{H} = \{H_n\}_{n\rightarrow\infty}$ be a sequence of hypergraphs and let $\pi = k_1 +
\dots + k_t$ and $\pi' = m_1+\dots+m_r$.    Let $\phi : \{1,\dots,r\} \rightarrow \{1,\dots,t\}$ be
the surjection for the refinement of $\pi'$ of $\pi$.  That is, $k_i = \sum_{j : \phi(j) = i} m_j$.
By Lemma~\ref{lem:partitetonormal}, we only need to show that \proppartexp{$\pi'$} holds, so
let $S'_1,\dots,S'_t$ be disjoint sets with $S'_i \subseteq \binom{V(H_n)}{m_i}$.  For $1 \leq i
\leq t$, define
\begin{align*}
  S_i = \{ X_{j_1} \cup \dots \cup X_{j_{\ell}} : \,\, \{j_1,\dots,j_{\ell}\} = \{ j : \phi(j) = i
  \} \,\, \text{and } \forall a, X_{j_a} \in S'_{j_a} \}.
\end{align*}
In other words, $S_i$ consists of all vertex sets formed by combining via the refinement sets from
$S'_1,\dots,S'_t$.  Since $S'_1,\dots,S'_t$ are disjoint,
\begin{align} \label{eq:piimpliespiprime}
  e(S_1,\dots,S_t) = e(S'_1,\dots,S'_r)
  \quad \quad \text{and} \quad \quad
  |S_1| \cdots |S_t| = |S'_1| \cdots |S'_r|.
\end{align}
Since \propexpand{$\pi$} holds for $\mathcal{H}$,
\begin{align*}
  e(S_1, \dots, S_t) = p |S_1| \cdots |S_t| + o(n^k).
\end{align*}
Combining this with \eqref{eq:piimpliespiprime} shows that \proppartexp{$\pi'$} holds for
$\mathcal{H}$.
\end{proof} 

\subsection{Clique Discrepency} 
\label{sub:implicationscd}

Our goal in this subsection is to discuss and prove all the implications in Table~\ref{tab:main}
involving \propcd{$\ell$}. In particular, Lemma~\ref{lem:cdtoexp} below states that \propcd{$\ell$}
$\Rightarrow$ \propexpand{$\pi$} if $\max \pi \leq \ell$.

The implication \propcdp{$\ell$} $\Rightarrow$ \propcdp{$\ell-1$} is easy to see directly from the
definitions: given an $(\ell-1)$-uniform hypergraph $G$, let $F$ be the $\ell$-uniform hypergraph
whose hyperedges consist of the $\ell$-cliques in $G$.  Then $\mathcal{K}_k(G) = \mathcal{K}_k(F)$,
so applying \propcdp{$\ell$} to $F$ implies that \propcdp{$\ell-1$} holds for $G$.

As part of their initial investigation of hypergraph quasirandomness, Chung and
Graham~\cite{hqsi-chung90-2} proved that \propcdh{$k-1$} $\Rightarrow$ \propdev{$k$}.  Since the
reverse implication also holds, these properties are equivalent.  Indeed, they have also both been
shown equivalent to \propcount{All}, the property that for every $k$-uniform hypergraph $F$, the
number of labeled copies of $F$ in $\mathcal{H}$ is $(1/2)^{|E(F)|} n^{|V(F)|} + o(n^{|V(F)|})$.

The final implication involving \propcdp{$\ell$} in Table~\ref{tab:main} is that if $\pi =
k_1+\dots+k_t$ is a proper partition of $k$ where $k_i \leq \ell$ for all $i$, then \propcdp{$\ell$}
$\Rightarrow$ \propexpandp{$\pi$}. 
\begin{lemma} \label{lem:cdtoexp}
  Let $k \geq 3$, let $2 \leq \ell < k$, and let $\pi$ be a proper partition of $k$
  where $\max \pi \leq \ell$.  Then \propcdp{$\ell$} $\Rightarrow$ \propexpandp{$\pi$}.
\end{lemma}

\newcounter{cdtopiclaimctr}
\newtheorem{cdtopiclaim}[cdtopiclaimctr]{Claim}

\begin{proof} 
First, view $\pi$ as an ordered partition $\vec{\pi} = (k_1,\dots,k_t)$ where $\sum k_i = k$.  Let
$\mathcal{H} = \{H_n\}_{n \rightarrow\infty}$ be a sequence of hypergraphs satisfying
\propcd{$\ell$}.  Throughout this proof, for notational simplicity we drop the subscript $n$.  By
Lemma~\ref{lem:partitetonormal}, we only need to show that \proppartexp{$\pi$} holds, so let $S_i
\subseteq \binom{V(H)}{k_i}$ be given such that $S_1,\dots,S_t$ are disjoint.  Define
\begin{align*}
  \mathcal{M} = \left\{ (m_1,\dots,m_t) : 0 \leq m_i \leq k, 
   \sum_{i=1}^t m_i = k \right\}.
\end{align*}
For $\vec{m} \in \mathcal{M}$, define the \emph{cliques of type $\vec{m}$} as the following set:
\begin{align*}
  \mathcal{T}_{\vec{m}} = \left\{ A \in \binom{V(S_1) \cup \dots \cup V(S_t)}{k} : 
  |A \cap V(S_i)| = m_i \,\,\text{and } \binom{A \cap V(S_i)}{k_i} \subseteq S_i \right\}.
\end{align*}
That is, the cliques of type $\vec{m}$ are the $k$-sets of vertices which have exactly $m_i$
vertices in $V(S_i)$ and if $m_i \geq k_i$ then all $k_i$ subsets of $A \cap V(S_i)$ are elements of
$S_i$  (since if $m_i < k_i$ then $\binom{A \cap V(S_i)}{k_i} = \emptyset$).  Depending on
$\vec{\pi}$, $k$, and $\ell$ some of the collections $\mathcal{T}_{\vec{m}}$ could be empty.  Now
define an equivalence relation $\sim$ on $\mathcal{M}$ as follows.  For $\vec{m}, \vec{m}' \in
\mathcal{M}$,
\begin{align*}
  \vec{m} \sim \vec{m}' \,\,\text{if and only if } \, 
  &\{ i : m_i < \ell \} = \{ i : m'_i < \ell \} \\
  &\text{and } \forall i \in \{ i : m_i < \ell \},  m_i = m'_i
\end{align*}
In other words, $\vec{m} \sim \vec{m}'$ if they are equal in coordinates which are smaller than
$\ell$ and have the same sum of coordinates at least $\ell$.  For example, if $\vec{\pi} = (3,3,2,2,2)$ and
$\ell = 3$, then $(6,4,1,1,0) \sim (5,5,1,1,0)$.  For $\vec{m} \in \mathcal{M}$, define $[\vec{m}] = \{
\vec{m}' \in \mathcal{M} : \vec{m} \sim \vec{m}' \}$.  It is trivial to see that this is an
equivalence relation on $\mathcal{M}$.

\begin{cdtopiclaim} \label{c:pialone}
  $[(k_1,\dots,k_t)] = \{ (k_1,\dots,k_t) \}$.
\end{cdtopiclaim}

\begin{proof} 
Assume that $\vec{m} \sim (k_1,\dots,k_t)$.  For indices $i$
where $k_i < \ell$, $m_i = k_i$ and for indices where $k_i = \ell$, $m_i \geq \ell = k_i$.  But
since $\sum m_i = \sum k_i = k$, we must have $m_i = k_i$ for all $i$.
\end{proof} 

Define $\mathcal{T}_{[\vec{m}]} = \bigcup_{\vec{m}' \sim \vec{m}} \mathcal{T}_{\vec{m}'}$.  Note
that $\mathcal{T}_{[(k_1,\dots,k_t)]} \cap E(H) = \mathcal{T}_{(k_1,\dots,k_t)} \cap E(H)$ is
exactly the set of edges we would like to count; $\mathcal{T}_{(k_1,\dots,k_t)}$ is isomorphic
to the collection of ordered tuples $(s_1,\dots,s_t)$ such that $s_i \in S_i$ since $S_1,\dots,S_t$
is disjoint.  Thus the following claim completes the proof.

\newcommand{\emod}{\mathclose{}/\mathopen{}}

\begin{cdtopiclaim} \label{c:tcapH}
  For all $[\vec{m}] \in \mathcal{M}\emod\sim$, $|\mathcal{T}_{[\vec{m}]} \cap E(H)| = p \left|
  \mathcal{T}_{[\vec{m}]} \right| + o(n^k)$.
\end{cdtopiclaim}

\begin{proof} 
The proof is by induction; define a partial order on the equivalence classes in $\mathcal{M}\emod\sim$
as follows: $[\vec{m}'] < [\vec{m}]$ if one of the following holds:
\begin{itemize}
  \item $|\{ i : m'_i = 0 \}| > |\{i : m_i = 0\}|$, or

  \item $|\{ i : m'_i = 0 \}| = |\{i : m_i = 0\}|$ and $\{i : m'_i \geq \ell \} \subsetneq \{ i
    : m_i \geq \ell \}$, or

  \item $|\{ i : m'_i = 0 \}| = |\{i : m_i = 0\}|$ and $\{i : m'_i \geq \ell \} = \{ i : m_i \geq \ell \}$
    and
    \begin{align*}
      \sum_{\substack{1 \leq i \leq t \\ m'_i < \ell}} m'_i < \sum_{\substack{1 \leq i \leq t \\ m_i
      < \ell}} m_i.
    \end{align*}
\end{itemize}
Note that the definition is well defined since any vector in $[\vec{m}]$ has the same set of indices
$i$ where $m_i = 0$ and the same set of indices where $m_i \geq \ell$.  We prove Claim~\ref{c:tcapH}
by induction on this partial order.  The base case proves the statement for all minimum elements in
the partial order and the inductive argument applies the claim only for elements smaller in the
partial order.

For the base case we consider vectors with exactly one non-zero $m_i$ which equals $k$ since the
total sum of the entries of $\vec{m}$ is $k$.  Note that the equivalence class of
$(0,\dots,0,k,0,\dots,0)$ has size one, so the base case is to show that
$|\mathcal{T}_{(0,\dots,0,k,0,\dots,0)} \cap E(H)| = p|\mathcal{T}_{(0,\dots,0,k,0,\dots,0)}| +
o(n^k)$.  Assume that $m_i = k$ and for $j \neq i$, $m_j = 0$.  Since \propcd{$\ell$} $\Rightarrow$
\propcd{$k_i$}, \propcd{$k_i$} holds for $\mathcal{H}$.  Now apply \propcd{$k_i$} to the
$k_i$-uniform hypergraph $S_i$.  By definition, $\mathcal{T}_{\vec{m}} = \mathcal{K}_k(S_i)$ since
both are the $k$-sets all of whose $k_i$-subsets are elements of $S_i$.  Therefore, \propcd{$k_i$}
applied to $S_i$ implies $\left| \mathcal{K}_k(S_i) \cap E(H) \right| = p \left| \mathcal{K}_k(S_i)
\right| + o(n^k)$.

For the inductive step, define an $\ell$-uniform hypergraph $W_{[\vec{m}]}$ as
follows.  The vertex set of $W_{[\vec{m}]}$ is the same as the vertex set of $H$.  The edge set is
\begin{align*}
  E(W_{[\vec{m}]}) = \Big\{ B \in \binom{V(S_1) \cup \dots \cup V(S_t)}{\ell} :
    &\forall i, \,\, m_i < \ell \Longrightarrow |B \cap V(S_i)| \leq m_i, \\
    \text{and } &\forall i, \,\, \binom{B \cap V(S_i)}{k_i} \subseteq S_i \Big\}.
\end{align*}
Note that the definition is well defined since any vector in $[\vec{m}]$ has the same entries for
indices smaller than $\ell$.

\begin{cdtopiclaim} \label{c:cliquesWsubsetT}
  $\mathcal{K}_k(W_{[\vec{m}]}) \subseteq \cup_{\vec{m}' \in \mathcal{M}} \mathcal{T}_{\vec{m}'}$.
\end{cdtopiclaim}

\begin{proof} 
Let $A \in \mathcal{K}_k(W_{[\vec{m}]})$ and define $m'_i = |A \cap V(S_i)|$.  Since $A \subseteq
V(S_1) \cup \dots \cup V(S_t)$ and $S_1,\dots,S_t$ are disjoint, $\sum m'_i = k$.  Lastly, pick any
$C \in \binom{A \cap V(S_i)}{k_i}$ and let $B$ be any $\ell$-subset of $A$ containing $C$.  Such a
$B$ exists since $\max \pi \leq \ell$.
Since $A$ is a clique of $W_{[\vec{m}]}$, $B$ is an edge of $W_{[\vec{m}]}$ which implies that all
$k_i$-subsets of $B \cap V(S_i)$ are elements of $S_i$.  But $C \subseteq B \cap V(S_i)$ so $C \in
S_i$, implying that $A \in \mathcal{T}_{\vec{m}'}$.
\end{proof} 

\begin{cdtopiclaim} \label{c:TmsubsetW}
  $\mathcal{T}_{[\vec{m}]} \subseteq \mathcal{K}_k(W_{[\vec{m}]})$.
\end{cdtopiclaim}

\begin{proof} 
Let $A \in \mathcal{T}_{[\vec{m}]}$ and let $B$ be any $\ell$-subset of $A$.  Then $|B \cap V(S_i)|
\leq |A \cap V(S_i)| = m_i$ for all $i$.  Also, for any $C \subseteq B \cap V(S_i)$ with $|C| =
k_i$, $C \subseteq A \cap V(S_i)$ so $C \in S_i$.
\end{proof} 

\begin{cdtopiclaim} \label{c:Tdisjoint}
  For $\vec{m} \neq \vec{m}'$, $\mathcal{T}_{\vec{m}} \cap \mathcal{T}_{\vec{m}'} = \emptyset$.
\end{cdtopiclaim}

\begin{proof} 
Let $A \in \mathcal{T}_{\vec{m}} \cap \mathcal{T}_{\vec{m}'}$.  Then $|A \cap V(S_i)| = m_i$ and $|A
\cap V(S_i)| = m'_i$ for all $i$ so $m_i = m'_i$ for all $i$.
\end{proof} 

\begin{cdtopiclaim} \label{c:Wispartition}
  There exists a collection $\mathcal{M}' \subseteq \mathcal{M}\emod\sim$ such that
  \begin{align*}
    \mathcal{K}_k(W_{[\vec{m}]}) = \mathcal{T}_{[\vec{m}]} \, \dot\cup \, \dot\bigcup_{[\vec{m}'] \in
    \mathcal{M}'} \mathcal{T}_{[\vec{m}']}.
  \end{align*}
\end{cdtopiclaim}

\begin{proof} 
By Claims~\ref{c:cliquesWsubsetT},~\ref{c:TmsubsetW}, and~\ref{c:Tdisjoint} we only need to prove
that for every $[\vec{m}'] \in \mathcal{M}\emod\sim$, either $\mathcal{K}_k(W_{[\vec{m}]})$ contains
$\mathcal{T}_{[\vec{m}']}$ or is disjoint from $\mathcal{T}_{[\vec{m}']}$.

Let $A_1, A_2 \in \mathcal{T}_{[\vec{m}']}$ such that $A_1 \in \mathcal{K}_k(W_{[\vec{m}]})$.  Let
$B_2$ be any $\ell$-subset of $A_2$.  We would like to show that $B_2$ is in $E(W_{[\vec{m}]})$ to
imply that $A_2 \in \mathcal{K}_k(W_{[\vec{m}]})$.  For $i$ with $m_i < \ell$, let $f_i = |B_2 \cap
V(S_i)|$ and let $B_1$ be an $\ell$-subset of $A_1$ which takes any $f_i$ elements of $A_1 \cap
V(S_i)$ for each $i$ with $m_i < \ell$ and takes vertices arbitrarily from $V(S_i)$ for $i$ where $m_i \geq
\ell$.  There exists such a set $B_1$ since for coordinates $i$ where $m_i < \ell$, $f_i \leq |A_2
\cap V(S_i)| = m_i = |A_1 \cap V(S_i)|$ so there exists a subset of $A_1 \cap V(S_i)$ of size $f_i$
and this subset can be used for $B_1 \cap V(S_i)$.  Also, once these vertices are picked, $B_1$ can
be extended to an $\ell$-set by taking vertices only from the other coordinates since $\sum \{ |A_1
\cap V(S_i)| : m_i \geq \ell \} = \sum \{ |A_2 \cap V(S_i)| : m_i \geq \ell \}$.  In addition, this
argument showing the existence of $B_1$ does not depend on the representatives $\vec{m}$ and
$\vec{m}'$ chosen for the equivalence classes $[\vec{m}]$ and $[\vec{m}']$.

Now that we have defined $B_1$, since $A_1$ is a clique of $W_{[\vec{m}]}$ and $B_1$ is an
$\ell$-subset of $A_1$, $B_1$ must be an element of $E(W_{[\vec{m}]})$.  This implies for $i$ with
$m_i < \ell$ that $f_i = |B_1 \cap V(S_i)| \leq m_i$ so that $|B_2 \cap V(S_i)| \leq m_i$.  Lastly,
if $C$ is a $k_i$-subset of $B_2 \cap V(S_i)$, then $C$ is a $k_i$-subset of $A_2 \cap V(S_i)$ so $C
\in S_i$.
\end{proof} 

The actual description of which equivalence classes are in $\mathcal{M}'$ is complicated and depends
on the relationships between $k_i$, $m_i$ and $\ell$.  Fortunately, we don't need the exact
description; we just require that every $[\vec{m}'] \in \mathcal{M}'$ appears below $[\vec{m}]$ in
the partial ordering.  Assume that $\mathcal{M}'$ is defined so that for each $[\vec{m}'] \in
\mathcal{M}'$, $\mathcal{T}_{[\vec{m}']} \neq \emptyset$.  Recall that it is possible for
$\mathcal{T}_{\vec{m}'}$ to be empty for certain $\vec{m}'$ depending on the interaction between the
hypergraph $H$, $\pi$, $k$, and $\ell$.  Therefore, in the remainder of this proof we just ignore
the collections $\mathcal{T}_{\vec{m}'}$ which are empty.

\begin{cdtopiclaim} \label{c:mprimebelow}
  For every $[\vec{m}'] \in \mathcal{M}'$, $[\vec{m}'] < [\vec{m}]$.
\end{cdtopiclaim}

\begin{proof} 
Let $A \in \mathcal{T}_{[\vec{m}']}$.  First, we prove that $\{i : m'_i = 0\} \supseteq \{ i : m_i =
0 \}$.  Assume for contradiction there exists some $i$ with $m_i' \neq 0$ and $m_i = 0$. Since $m'_i
\neq 0$, $A$ contains a vertex $x$ inside $V(S_i)$.  But now let $B$ be any $\ell$-subset of $A$
containing $x$.  Since $m_i = 0$ and $x \in V(S_i)$, this $\ell$-subset $B$ is not in
$E(W_{[\vec{m}]})$ so $A$ is not a $k$-clique of $W_{[\vec{m}]}$ contradicting $[\vec{m}'] \in
\mathcal{M}'$. If $\{i : m'_i = 0\} \supsetneq \{ i : m_i = 0 \}$, then $[\vec{m}'] < [\vec{m}]$.
Therefore, assume that $\{i : m'_i = 0\} = \{ i : m_i = 0 \}$.

Next, we prove that for $i$ with $m_i < \ell$, $m'_i \leq m_i$. Assume for contradiction that there
exists an $i$ such that $m'_i > m_i$ and $m_i < \ell$.  Let $A \in \mathcal{T}_{[\vec{m}']}$.  Then
$|A \cap V(S_i)| = m'_i \geq m_i + 1$ so let $B$ be an $\ell$-subsets of $A$ which has at least $m_i
+ 1$ elements of $A \cap V(S_i)$.  There exists such a $B$ since $m_i + 1 \leq \ell$ and $m_i + 1
\leq |A \cap V(S_i)|$.  But now $B$ is not in $E(W_{[\vec{m}]})$ since $|B \cap V(S_i)| > m_i$ and
$m_i < \ell$ and this contradicts that $[\vec{m}'] \in \mathcal{M}'$.

Since for every $i$ with $m_i < \ell$, $m'_i \leq m_i$ we must have $\{ i : m'_i \geq \ell \}
\subseteq \{ i : m_i \geq \ell \}$.  If $\{ i : m'_i \geq \ell \} \subsetneq \{ i : m_i \geq \ell
\}$, then $[\vec{m}'] < [\vec{m}]$.  Therefore, assume that $\{ i : m'_i \geq \ell \} = \{ i : m_i
\geq \ell \}$.  This implies that
\begin{align} \label{eq:mprimebelow}
  \sum_{\substack{1 \leq i \leq t \\ m'_i < \ell}} m'_i \leq \sum_{\substack{1 \leq i \leq t \\ m_i <
  \ell}} m_i
\end{align}
since $m_i < \ell$ if and only if $m'_i < \ell$ and for these indices, $m'_i \leq m_i$.  If
\eqref{eq:mprimebelow} is a strict inequality, then $[\vec{m}'] < [\vec{m}]$ so assume that
\eqref{eq:mprimebelow} is an equality which implies $m_i = m'_i$ for all $i$ with $m_i < \ell$.
Combining this with $\{ i : m'_i \geq \ell \} = \{ i : m_i \geq \ell \}$ implies that $\vec{m} \sim
\vec{m}'$ which is a contradiction, since the union in Claim~\ref{c:Wispartition} is a disjoint
union.
\end{proof} 

Claims~\ref{c:Wispartition} and~\ref{c:mprimebelow} combine to finish the proof of
Claim~\ref{c:tcapH}.  By induction, we know the size of $\mathcal{T}_{[\vec{m}']} \cap E(H)$ for all
$[\vec{m}'] \in \mathcal{M}'$ and \propcd{$\ell$} implies that $|\mathcal{K}_k(W_{[\vec{m}]}) \cap
E(H)| = p |\mathcal{K}_k(W_{[\vec{m}]})| + o(n^k)$. A simple subtraction counts the size of
$\mathcal{T}_{[\vec{m}]} \cap V(H)$ as follows:
\begin{align*}
  \left| \mathcal{T}_{[\vec{m}]} \cap E(H) \right| 
  &= \left| \mathcal{K}_{k}(W_{[\vec{m}]}) \cap E(H) \right| - \sum_{[\vec{m}'] \in \mathcal{M}'} \left|
     \mathcal{T}_{[\vec{m}']} \cap E(H) \right|  \\
  &= p\left| \mathcal{K}_{k}(W_{[\vec{m}]}) \right| - \sum_{[\vec{m'}]\in \mathcal{M}'} p\left|
     \mathcal{T}_{[\vec{m}']} \right| + o(n^k) \\
  &= p\left( \left| \mathcal{K}_{k}(W_{[\vec{m}]}) \right| - \sum_{[\vec{m'}] \in \mathcal{M}'} \left|
     \mathcal{T}_{[\vec{m}']} \right|\right) + o(n^k) \\
  &= p \left| \mathcal{T}_{[\vec{m}]} \right| + o(n^k).
\end{align*}
\end{proof} 
By Claims~\ref{c:pialone} and~\ref{c:tcapH}, the proof of the lemma is now complete.
\end{proof} 

\subsection{Deviation} 
\label{sub:implicationsdev}

In this section, we discuss the implications involving \propdev{$\ell$} in Table~\ref{tab:main}.
The implications \propdev{$\ell$} $\Rightarrow$ \propdev{$\ell-1$} and \propdev{$\ell$}
$\Rightarrow$ \propcd{$\ell-1$} were both proved by Chung~\cite{hqsi-chung90}.  The remaining
implication is \propdev{$\ell$} $\Rightarrow$ \propexpand{$\pi$} for all $\ell$ and $\pi$. The proof
uses several similar techniques to the other deviation implications proved by
Chung~\cite{hqsi-chung90}.

\newcommand{\oct}[1]{\mathcal{O}[#1]}
\newcommand{\octs}[1]{\tilde{\mathcal{O}}[#1]}

\begin{definition}
  Let $A_1,\dots,A_k \subseteq V(H)$ be subsets of vertices such that $|A_i| \in \{1, 2\}$.  Define
  \begin{align*}
    \oct{A_1;\dots;A_k} = \left\{ (x_1,\dots,x_k) \in V(H)^k : x_i \in A_i \right\}.
  \end{align*}
  That is, $\oct{A_1;\dots;A_k}$ is the collection of tuples of the squashed octahedron using the
  vertices from $A_1, \dots, A_k$.  Next, define
  \begin{align*}
    \octs{A_1;\dots;A_k} = \Big\{ \{x_1,\dots,x_k\} : (x_1,\dots,x_k) \in \oct{A_1;\dots;A_k} \,\,
    \text{and } |\{x_1,\dots,x_k\}| = k \Big\}
  \end{align*}
  so that $\octs{A_1;\dots;A_k}$ are the $k$-sets which come from tuples of distinct vertices of the
  squashed octahedron.  Lastly, define
  \begin{align*}
    \eta_H(A_1;\dots;A_k) = \begin{cases}
      1, \,&\text{if } |\octs{A_1;\dots;A_k} \cap E(H)| \,\, \text{is even},\\
      -1,   &\text{otherwise}
    \end{cases}
  \end{align*}
  For notational convenience, the braces defining $A_i$ are usually dropped.  For example, we will
  write $\oct{x;y_0,y_1;z_0,z_1}$ for $\oct{\{x\};\{y_0,y_1\};\{z_0,z_1\}}$.
\end{definition}

\begin{definition}
  Let $P \subseteq V(H)^k$, let $0 \leq \ell \leq k$, and let $H$ be a $k$-uniform hypergraph.
  Define
  \begin{align*}
    \dev_{\ell,P}(H) := 
    \sum_{\substack{x_1,\dots,x_{k-\ell},y_{1,0},y_{1,1},\dots,y_{\ell,0},y_{\ell,1} \in V(H) \\
      \oct{x_1;\dots;x_{k-\ell};y_{1,0},y_{1,1};\dots;y_{\ell,0},y_{\ell,1}} \subseteq P}}
      \eta_H(x_1;\dots;x_{k-\ell};y_{1,0},y_{1,1};\dots;y_{\ell,0},y_{\ell,1}).
  \end{align*}
  and let $\dev_\ell(H) := \dev_{\ell,V(H)^k}(H)$.
\end{definition}

Note that by definition, \propdev{$\ell$} is the property that $\dev_{\ell}(H_n) = o(n^{k+\ell})$.
Also, \propcdh{$\ell$} is the property that for all $\ell$-uniform hypergraphs $G$, $\dev_{0,P}(H_n)
= o(n^k)$, where $P$ is the collection of tuples which form $k$-cliques of $G$.  Finally,
\propexpandh{$k_1 + \dots + k_t$} is the property that for all $S_i \subseteq \binom{V(H)}{k_i}$,
$\dev_{0,P}(H_n) = o(n^k)$ where $P$ is now the collection of $k$-tuples which are formed by taking
one element of $S_i$ for each $i$.

\begin{definition}
  A set $P \subseteq V(H)^k$ is called \emph{complete in coordinate $i$} if there exists a $P'
  \subseteq V(H)^{k-1}$ such that $P = \{ (x_1,\dots,x_k) : (x_1,\dots,x_{i-1},x_{i+1},\dots,x_k)
  \in P', x_i \in V(H)\}$.
\end{definition}

The following two lemmas are the heart of Chung's~\cite{hqsi-chung90} proof that \propdev{$\ell$}
$\Rightarrow$ \propcdh{$\ell-1$}, although the lemmas aren't stated separately; they appear
implicitly in the proof.

\begin{lemma} \label{lem:subdev} (Chung~\cite{hqsi-chung90})
  Let $H$ be a $k$-uniform hypergraph, let $P,Q \subseteq V(H)^k$, and let $2 \leq \ell \leq k$.  If
  $Q$ is complete in coordinate $i$ where $k-\ell+1 \leq i \leq k$, then
  \begin{align*}
    \dev_{\ell,P\cap Q}(H) \leq \dev_{\ell,P}(H).
  \end{align*}
\end{lemma}

\begin{lemma} \label{lem:devcauchy} (Chung~\cite{hqsi-chung90})
  Let $1 \leq \ell \leq k$, let $P \subseteq V(H)^k$, and let $\mathcal{H} =
  \{H_n\}_{n\rightarrow\infty}$ be a sequence of hypergraphs with $|V(H_n)| = n$.  If
  $\dev_{\ell,P}(H_n) = o(n^{k+\ell})$, then $\dev_{\ell-1,P}(H_n) = o(n^{k+\ell-1})$.
\end{lemma}

The fact that \propdev{$\ell$} $\Rightarrow$ \propdev{$\ell-1$} follows from
Lemma~\ref{lem:devcauchy} for $P = V(H)^k$.  The fact that \propdev{$\ell$} $\Rightarrow$
\propcdh{$\ell-1$} follows from a combination of Lemmas~\ref{lem:subdev} and~\ref{lem:devcauchy}:
given an $(\ell-1)$-uniform hypergraph $G$, define $P$ to be the tuples which are $k$-cliques in $G$
and write $P$ as an intersection of sets complete in a coordinate (for details, see
Lemma~\ref{lem:devtocd}).  Combining Lemmas~\ref{lem:subdev} and~\ref{lem:devcauchy} in a slightly
different way proves that \propdev{$\ell$} $\Rightarrow$ \propexpandh{$\pi$}.

\begin{lemma} \label{lem:devtoexp}
  For all $2 \leq \ell \leq k$ and all proper partitions $\pi$, \propdev{$\ell$} $\Rightarrow$ \propexpandh{$\pi$}.
\end{lemma}

\begin{proof} 
Since \propdev{$\ell$} $\Rightarrow$ \propdev{$\ell-1$}, by Lemma~\ref{lem:exptoexp} we just need to
prove that \propdev{$2$} $\Rightarrow$ \propexpand{$k_1 + k_2$} for every $k_1, k_2 \geq 1$ with $k_1 + k_2
= k$.  Indeed, every $\pi$ is a refinement of $k_1 + k_2$ for some choice of $k_1$ and $k_2$.  Given
$S_1 \subseteq \binom{V(H)}{k_1}$ and $S_2 \subseteq \binom{V(H)}{k_2}$, define
\begin{align*}
  P_1 &= \left\{ (x_1,\dots,x_{k-2},y,z) \in V(H)^k : \{x_1,\dots,x_{k_1-1},y\} \in S_1 \right\}, \\
  P_2 &= \left\{ (x_1,\dots,x_{k-2},y,z) \in V(H)^k : \{x_{k_1},\dots,x_{k-2},z\} \in S_2 \right\}.
\end{align*}
$P_1$ is complete in coordinate $k$ as we can let $P_1' = \{ (x_1,\dots,x_{k-2},y) :
\{x_1,\dots,x_{k_1-1},y\} \in S_1\}$ and similarly $P_2$ is complete in coordinate $k-1$.  Thus by
Lemma~\ref{lem:subdev} and the fact that \propdev{$2$} holds,
\begin{align*}
  \dev_{2,P_1 \cap P_2}(H_n) \leq \dev_{2,V(H)^k}(H_n) = o(n^{k+2}).
\end{align*}
Now apply Lemma~\ref{lem:devcauchy} to show that
\begin{align} \label{eq:devtoexplittleo}
  \dev_{0,P_1 \cap P_2}(H_n) = o(n^k).
\end{align}
But by definition,
\begin{align}
  \dev_{0,P_1 \cap P_2}(H_n) 
  &= \sum_{\substack{x_1,\dots,x_k \in V(H) \\ \oct{x_1;\dots;x_k} \subseteq P_1 \cap P_2}}
  \eta(x_1;\dots;x_k) \nonumber \\
  &= \sum_{(x_1,\dots,x_k) \in P_1 \cap P_2} \eta(x_1;\dots;x_k) \label{eq:devtoexp}
\end{align}
Let $(x_1,\dots,x_k) \in P_1 \cap P_2$, let $s_1 = \{x_1,\dots,x_{k_1-1},x_{k-1}\}$ and let $s_2 =
\{x_{k_1},\dots,x_{k-2},x_{k}\}$.  Since $(x_1,\dots,x_k) \in P_1 \cap P_2$, we have $s_1 \in S_1$
and $s_2 \in S_2$.  Also, $\eta(x_1;\dots;x_k) = -1$ if $s_1 \cup s_2 \in E(H)$ and is $1$
otherwise.  Each $(s_1,s_2) \in S_1 \times S_2$ is counted $k_1!k_2! $ times in \eqref{eq:devtoexp},
since the number of $(x_1,\dots,x_k) \in P_1 \cap P_2$ with $\{x_1, \dots, x_{k_1-1}, x_{k-1}\} =
s_1$ and $\{x_{k_1},\dots,x_{k-2},x_k\} = s_2$ is $k_1!k_2! $.  Thus equations
\eqref{eq:devtoexplittleo} and \eqref{eq:devtoexp} combine to show that
\begin{align*}
  k_1! k_2! \Big( \left|\{ (s_1,s_2) \in S_1 \times S_2 : s_1 \cup s_2 \not\in E(H)\}\right| 
  -\left|\{ (s_1,s_2) \in S_1 \times S_2 : s_1 \cup s_2 \in E(H)\}\right|
  \Big)
\end{align*}
is $o(n^k)$.  Since $k_1$ and $k_2$ are constants, this implies that the number of tuples $(s_1,s_2)
\in S_1 \times S_2$ which are edges of $H$ is $\frac{1}{2}$ of all such tuples (up to $o(n^k)$), so
\propexpand{$k_1 + k_2$} holds and the proof is complete.
\end{proof} 

\section{Constructions} 
\label{sec:constructions}

To show a property $P$ does not imply a property $Q$, we construct a hypergraph sequence that
satisfies $P$ but fails $Q$.  While Table~\ref{tab:main} states several results of this form, we
only use three constructions.  This section defines these constructions and proves several facts
about them. The constructions are built from random graphs and random hypergraphs, so we are
actually defining three probability distributions over $n$-vertex, $k$-uniform hypergraphs which we
call $A_{\ell}(n,p)$, $B_{\vec{\pi}}(n,p)$ and $D(n,1/2)$.  As is typical in random graph theory, we
will abuse notation by also writing $A_{\ell}(n,p)$, $B_{\vec{\pi}}(n,p)$ and $D(n,1/2)$ for a
particular hypergraph drawn from these distributions. Most likely, these constructions can be
made explicit by replacing the use of the random hypergraph with a quasirandom hypergraph.

Before stating the constructions, we briefly state two well known concentration bounds on sums of
indicator random variables.  For more details, see~\cite{prob-alonspencer}.

\begin{lemma} \label{lem:chebyshev} \textbf{(Chebyshev's Inequality)}
  Let $X_1, \dots, X_n$ be indicator random variables, let $X = \sum X_i$, and let $\mu =
  \mathbb{E}[X]$ and $\sigma^2 = Var(X) = \mathbb{E}[(X-\mu)^2]$.  For every $\epsilon > 0$,
  \begin{align*}
    \mathbb{P}\left[ |X - \mu| > \epsilon \mu \right] \leq \frac{\sigma^2}{\epsilon^2 \mu^2}.
  \end{align*}
\end{lemma}

\begin{lemma} \label{lem:chernoff} \textbf{(Chernoff Bound)}
  Let $0 < p < 1$, let $X_1,\dots,X_n$ be mutually independent indicator random variables with
  $\mathbb{P}[X_i = 1] = p$ for all $i$, let $X = \sum X_i$, and let $\mu = \mathbb{E}[X] = pn$.
  Then for all $a > 0$,
  \begin{align*}
    \mathbb{P}[ \left| X - \mu \right| > a] \leq 2 e^{-a^2/2n}.
  \end{align*}
\end{lemma}

\newtheorem*{defofA}{Construction of $A_{\ell}(n,p)$}
\begin{defofA}
  For $n \in \mathbb{N}$, $2 \leq \ell \leq k-1$, and $0 < p < 1$ with $p \in \mathbb{Q}$ so $p =
  \frac{a}{b}$ with $a,b \in \mathbb{Z}^+$, define a probability distribution $A_{\ell}(n,p)$ on
  $k$-uniform, $n$-vertex hypergraphs as follows.  Let $c : E(K^{(\ell)}_n) \rightarrow
  \{0,\dots,b-1\}$ be a random $b$-coloring of the edges of the complete $\ell$-uniform hypergraph
  $K^{(\ell)}_n$ where each edge receives each color with equal probability independently of all
  other edges.  Let the vertex set of $A_\ell(n,p)$ be $V(K^{(\ell)}_n)$ and make $W \subseteq
  V(A_\ell(n,p))$ a hyperedge of $A_\ell(n,p)$ if $|W| = k$ and
  \begin{align*}
    \sum_{\{x_1,\dots,x_\ell\} \subseteq W} c(\{x_1,\dots,x_\ell\}) < a \pmod b.
  \end{align*}
\end{defofA}

\begin{lemma} \label{lem:Aispdense}
  For every $\epsilon > 0$, with probability going to one as $n$ goes to infinity,
  \begin{align*}
    \left||E(A_\ell(n,p))| - p \binom{n}{k}\right| \leq \epsilon n^k.
  \end{align*}
\end{lemma}

\begin{proof} 
Let $W \subseteq V(A_\ell(n,p))$ with $|W| = k$ and let $\{x_1,\dots,x_\ell\} \subseteq W$. Define
\begin{align*}
  \Delta = \sum_{\substack{\{y_1,\dots,y_\ell\} \subseteq W \\ 
  \{x_1,\dots,x_\ell\} \neq \{y_1,\dots,y_\ell\}} } c(\{y_1,\dots,y_\ell\}).
\end{align*}
Now thinking of $\Delta$ as fixed, there are exactly $a$ choices for $c(\{x_1,\dots,x_\ell\})$ such
that $\Delta + c(\{x_1,\dots,x_\ell\}) < a \pmod b$.  Since the edge $\{x_1,\dots,x_\ell\}$ receives
each color with equal probability, the probability that $W$ is an edge of $A_\ell(n,p)$ is
$\frac{a}{b} = p$.  Therefore, the expected number of edges of $A_\ell(n,p)$ is $p \binom{n}{k}$.

By the second moment method, with probability going to one as $n$ goes to infinity, $\left|
|E(A_\ell(n,p))| - p\binom{n}{k}\right| \leq \epsilon n^k$.  Indeed, for each $k$-set $W$ in
$V(A_\ell(n,p))$, define an indicator random variable $X_W$ where $X_W = 1$ if $W$ is an edge of
$A_\ell(n,p)$. Let $X = \sum X_W$ so that $X = |E(A_\ell(n,p))|$ and $\mu = \mathbb{E}[X] = p
\binom{n}{k}$. Let $\hat\epsilon = \frac{\epsilon}{p}$ so that $|X - \mu| \leq \hat\epsilon \mu$
implies that $|E(A_\ell(n,p)) - p\binom{n}{k}| \leq \epsilon n^k$. Since $X$ is the sum of indicator
random variables, the variance $Var(X) = \sum_W Var(X_W) + 2\sum_{W,W'} Cov(X_W,X_{W'})$.  The event
``$X_W = 1$'' will depend on ``$X_{W'} = 1$'' if and only if $W$ and $W'$ intersect in at least
$\ell$ vertices, so there are at most $n^{2k-1}$ dependent pairs $(X_W, X_{W'})$.  This implies that
there are at most $n^{2k-1}$ pairs $(X_W,X_{W'})$ with $Cov(X_W,X_{W'}) \neq 0$ so that $Var(X) =
o(n^{2k})$.  Since $\mu^2 = \Omega(n^{2k})$, Chebyshev's Inequality (Lemma~\ref{lem:chebyshev})
implies that $\mathbb{P}\left[ |X - \mu| > \hat\epsilon \mu \right] \rightarrow 0$ as $n \rightarrow
\infty$, completing the proof.  For more details on the second moment method,
see~\cite{prob-alonspencer}.
\end{proof} 

\newtheorem*{defofB}{Construction of $B_{\vec{\pi}}(n,p)$}
\begin{defofB}
  Let $\vec{\pi} = (k_1,\dots,k_t)$ be a proper ordered partition of $k$, let $n \in \mathbb{N}$,
  and let $0 < p < 1$ with $p \in \mathbb{Q}$ so $p = \frac{a}{b}$ with $a,b \in \mathbb{Z}^+$.
  Define a probability distribution $B_{\vec{\pi}}(n,p)$ on $k$-uniform hypergraphs with vertex set
  $[n]$ as follows.  For $1 \leq i \leq t$, let $c_i : \binom{[n]}{k_i} \rightarrow \{0,\dots,b-1\}$
  be a random $b$-coloring of the hyperedges of the complete $n$-vertex, $k_i$-uniform hypergraph
  where each hyperedge receives each color with equal probability independently.  Form a $k$-uniform
  hypergraph $B_{\vec{\pi}}(n,p)$ on vertex set $[n]$ as follows.  Let $W \subseteq [n]$ with $|W| =
  k$ and partition $W$ into $W_1,\dots,W_t$ such that $|W_j| = k_j$ and for all $j < \ell$, every
  element of $W_j$ is smaller than every element of $W_{\ell}$.  In other words, $W_1$ is the set of first
  $k_1$ vertices of $W$ in the ordering, $W_2$ is the set of next $k_2$ vertices, and so on.  Make $W$ a
  hyperedge of $B_{\vec{\pi}}(n,p)$ if
  \begin{align*}
    \sum_{i=1}^t c_i(W_i) < a \pmod b.
  \end{align*}
\end{defofB}

\begin{lemma} \label{lem:Bispdense}
  For every $\epsilon > 0$, with probability going to one as $n$ goes to infinity,
  \begin{align*}
    \left||E(B_{\vec{\pi}}(n,p))| - p \binom{n}{k}\right| \leq \epsilon n^k.
  \end{align*}
\end{lemma}

\begin{proof} 
Let $W \in \binom{[n]}{k}$ and let $W_1,\dots,W_t$ be the partition of $W$ as in the construction.
Let $\Delta = \sum_{i=1}^{t-1} c(W_i)$.  There are exactly $a$ choices for $c(W_t)$ such that
$\Delta + c(W_t) < a \pmod b$ so the probability that $W$ is a hyperedge is $\frac{a}{b} = p$.
Since two $k$-sets will depend on each other only if they share at least one vertex, the second
moment method implies that with high probability, $|E(B_{\vec{\pi}}(n,p))| = p \binom{n}{k} \pm
\epsilon n^k$.
\end{proof} 

\newtheorem*{defofD}{Construction of $D(n,\frac{1}{2})$}
\begin{defofD}
  For $k \geq 3$ and $n \in \mathbb{N}$, define a probability distribution $D(n,\frac{1}{2})$ on
  $k$-uniform, $n$-vertex hypergraphs as follows.  Let $G = G^{(k-1)}(n,\frac{1}{2})$ be the random
  $(k-1)$-uniform hypergraph with edge probability $\frac{1}{2}$.  For each $T \in
  \binom{V(D(n,\frac{1}{2}))}{k}$, select a $(k-2)$-subset of $T$ uniformly at random from among all
  $(k-2)$-subsets of $T$.  Call the chosen $(k-2)$-subset of $T$ the \emph{head} of $T$ and note
  that the choice of the head of $T$ is selected independently of all other choices for heads for
  other $k$-sets.  If $T = \{x_1,\dots,x_{k-2},y,z\}$ where $\{x_1,\dots,x_{k-2}\}$ is the head,
  make $T$ a hyperedge of $D(n,\frac{1}{2})$ if either both or neither of
  $\{x_1,\dots,x_{k-2},y\},\{x_1,\dots,x_{k-2},z\}$ are edges of $G$.
\end{defofD}

\begin{lemma} \label{lem:Dispdense}
  For every $\epsilon > 0$, with probability going to one as $n$ goes to infinity,
  \begin{align*}
    \left||E(D(n,1/2))| - \frac{1}{2} \binom{n}{k}\right| \leq \epsilon n^k.
  \end{align*}
\end{lemma}

\begin{proof} 
Let $T \in \binom{V(D(n,1/2))}{k}$.  Conditioning on the choice of head $\{x_1,\dots,x_{k-2}\}$ and
the behavior of $\{x_1,\dots,x_{k-2},y\}$ in $G$, the set $\{x_1,\dots,x_{k-2},z\}$ is a hyperedge
of $G$ with probability $\frac{1}{2}$, so the probability that $T$ is a hyperedge of $D(n,1/2)$ is
$\frac{1}{2}$.  Since two $k$-sets will depend on each other only if they share at least $k-1$
vertices, the second moment method implies that $|E(D(n,1/2))| = \frac{1}{2} \binom{n}{k} \pm
\epsilon n^k$ with high probability.
\end{proof} 

The next few sections prove that with high probability, $A(n,p)$, $B_{\vec{\pi}}(n,p)$, and
$D(n,\frac{1}{2})$ satisfy and fail the following properties.

\begin{itemize}
  \item $A_{\ell}(n,p)$
    \begin{itemize}
      \item Satisfies: \propexpandp{$\pi$} for all $\pi$, \propcdp{$\ell-1$}, and \propdev{$\ell$}.
      \item Fails: \propcdp{$\ell$} and \propdev{$\ell+1$}.
    \end{itemize}
  \item $B_{\vec{\pi}}(n,p)$
    \begin{itemize}
      \item Satisfies: \propexpandp{$\pi'$} for $\pi \not\leq \pi'$ and 
        \propcdp{$\ell$} for $\ell < \max \pi$.
      \item Fails: \propexpandp{$\pi$} and \propdev{$2$}
    \end{itemize}
  \item $D(n,\frac{1}{2})$
    \begin{itemize}
      \item Satisfies: \propexpandp{$\pi$} for all $\pi$,
      \item Fails: \propdev{$2$}
    \end{itemize}
\end{itemize}

\subsection{Failure of quasirandom properties} 
\label{sub:constrfails}

\begin{lemma} \label{lem:constrAfailsCD}
  ($A_{\ell}(n,p)$ fails \propcd{$\ell$}) For $2 \leq \ell \leq k -1$, with probability going to one
  as $n$ goes to infinity, there exists an $\ell$-uniform hypergraph $G$ on vertex set
  $V(A_\ell(n,p))$ such that
  \begin{align*}
    \Big| |\mathcal{K}_k(G) \cap A_\ell(n,p)| - p |\mathcal{K}_k(G)| \Big| >
    \frac{1-p}{2 b^{k^\ell}} \binom{n}{k}.
  \end{align*}
\end{lemma}

\begin{proof} 
Let $G$ be the $\ell$-uniform hypergraph with vertex set $V(A_{\ell}(n,p))$ and edge set the set of
edges of $K^{(\ell)}_n$ colored zero in the definition of $A_\ell(n,p)$.  With high probability, the
second moment method implies that the number of $k$-cliques in $G$ is $b^{-\binom{k}{\ell}}
\binom{n}{k} + o(n^k)$.  By definition, $A_\ell(n,p)$ will intersect all of the $k$-cliques of $G$
so
\begin{align*}
    \Big| |\mathcal{K}_k(G) \cap A_\ell(n,p)| - p |\mathcal{K}_k(G)| \Big| = (1-p)
    |\mathcal{K}_k(G)| = (1-p) b^{-\binom{k}{\ell}} \binom{n}{k} + o(n^k)
\end{align*}
with high probability.
\end{proof} 

\begin{lemma} \label{lem:constrAfailsdev}
  ($A_{\ell}(n,\frac{1}{2})$ fails \propdev{$\ell+1$}) For $2 \leq \ell \leq k - 1$, there exists a
  constant $C > 0$ such that $\dev_{\ell+1}\left(A_\ell(n,1/2)\right) > C n^{k+\ell+1}$.
\end{lemma}

\begin{proof} 
We will prove that every non-degenerate squashed octahedron induces an even number of hyperedges of
$A_\ell(n,1/2)$.  Let $x_1, \dots, x_{k-\ell-1}, y_{1,0}, y_{1,1}, \dots, \linebreak[1] y_{\ell+1,0},
y_{\ell+1,1} \in V(A_\ell(n,1/2))$ be distinct vertices.  We claim that $|\mathcal{O}[x_1; \dots;
x_{k-\ell-1};y_{1,0},y_{1,1};\dots;\linebreak[1]y_{\ell+1,0},y_{\ell+1,1}] \cap E(H)|$ is always even.
Define $P_1 = \{x_1\}, \dots, P_{k-\ell-1} = \{x_{k-\ell-1}\}, P_{k-\ell} = \{y_{1,0}, y_{1,1}\},
\dots, P_k = \{y_{\ell+1,0},y_{\ell+1,1}\}$ so that $P_1,\dots,P_k$ are the parts of the squashed
octahedron.  Let $c : \binom{V(A_\ell(n,1/2))}{\ell} \rightarrow \{0,1\}$ be the random coloring
used in the definition of $A_\ell(n,1/2)$.  For a $k$-set $T$, define
\begin{align*}
  c(T) = \sum_{\substack{Z \subseteq T \\ |Z| = \ell}} c(Z) \pmod 2.
\end{align*}
Lastly, define $\mathcal{T}$ to be the collection of $k$-sets which take exactly one vertex from
each $P_i$.

\medskip
\noindent\textbf{Claim}: $\sum_{T \in \mathcal{T}} c(T) = 0 \pmod 2$.
\medskip

\begin{proof} 
Expand the definition of $c(T)$ to obtain
\begin{align} \label{eq:Adevweightsum}
  \sum_{T \in \mathcal{T}} c(T) = \sum_{T \in \mathcal{T}} \sum_{\substack{Z \subseteq T \\ |Z| =
  \ell}} c(Z) \pmod{2}.
\end{align}
Let $\Gamma_Z = \{k-\ell \leq i \leq k : Z \cap P_i = \emptyset\}$ and notice that $c(Z)$ appears
$2^{|\Gamma_Z|}$ times in \eqref{eq:Adevweightsum}.  Indeed, to form a $k$-set $T$ containing $Z$, there
is a choice between $y_{i,0}$ and $y_{i,1}$ for each $i \in \Gamma_Z$.  Since there are $\ell+1$
parts with two vertices and $|Z| = \ell$, $|\Gamma_Z| \geq 1$.  This implies that each $c(Z)$
appears an even number of times in \eqref{eq:Adevweightsum}, finishing the proof of the claim.
\end{proof} 

By definition, $T$ is a hyperedge of $A_\ell(n,1/2)$ if and only if $c(T) = 0 \pmod{2}$.  Thus the claim
implies that the number of $T$s which are not hyperedges is even, but since the squashed octahedron
has an even number of edges total, the number of $T$s which are hyperedges is then also even.  Thus
for every squashed octahedron using distinct vertices, the number of hyperedges appearing is even.
There are $(k+\ell+1)! \binom{n}{k+\ell+1}$ squashed octahedrons using distinct vertices and the
number of degenerate squashed octahedrons is $o(n^{k+\ell+1})$, completing the proof of the lemma.
\end{proof} 

\begin{lemma} \label{lem:constrBfailsexp}
  ($B_{\vec{\pi}}(n,p)$ fails \propexpand{$\pi$})  For all ordered partitions $\vec{\pi}$ of $k$,
  with probability going to one as $n$ goes to infinity, there exists $S_1 \subseteq
  \binom{[n]}{k_1}, \dots, S_t \subseteq \binom{[n]}{k_t}$ such that
  \begin{align*}
    \Big| e(S_1,\dots,S_t) - p |S_1| \cdots |S_t| \Big| >
    \frac{1}{2} \binom{k}{k_1,\dots,k_t} \frac{p}{b^{t} t^{k}} \binom{n}{k}.
  \end{align*}
\end{lemma}

\begin{proof} 
Divide $V(B_{\vec{\pi}}(n,p)) = [n]$ into $t$ almost equal parts $X_1 = \{1,\dots,\left\lfloor
\frac{n}{t}\right\rfloor\}$, $X_2 = \{\left\lfloor \frac{n}{t} \right\rfloor + 1, \dots,
\left\lfloor\frac{2n}{t}\right\rfloor\}$, and so on.  For $1 \leq i \leq t- 1$, let $S_i \subseteq
\binom{X_i}{k_i}$ be the set of hyperedges on $X_i$ colored zero under $c_i$ in the definition of
$B_{\vec{\pi}}(n,p)$.  Let $S_t \subseteq \binom{X_t}{k_t}$ be the set of hyperedges on $X_t$
colored $a$ under $c_t$ in the definition of $B_{\vec{\pi}}(n,p)$.

A $k$-set formed by taking a $k_i$-set from $S_i$ for each $i$ has color sum $a$, so is not a
hyperedge of $B_{\vec{\pi}}(n,p)$.  Thus $e(S_1,\dots,S_t) = 0$.  The second moment method implies
that with high probability $|S_i| = \frac{1}{b}\binom{n/t}{k_i} + o(n^{k_i})$. Therefore,
\begin{align*}
  \left| e(S_1,\dots,S_t) - p \prod_{i=1}^k |S_i| \right|=
  p \prod_{i=1}^k |S_i| &= \frac{p}{b^t} \prod_{i=1}^k \binom{n/t}{k_i} + o(n^k) \\
   &= \binom{k}{k_1,\dots,k_t} \frac{p}{b^{t} t^{k}} \binom{n}{k} + o(n^k)
\end{align*}
with high probability, completing the proof.
\end{proof} 

\begin{lemma} \label{lem:constrBfailsdev}
  ($B_{(k-1,1)}(n,1/2)$ fails \propdev{$2$}) Fix $k \geq 3$ and let $\vec{\pi} = (k-1,1)$. There
  exists a constant $C > 0$ such that with probability going to one as $n$ goes to infinity,
  \begin{align*}
    \dev_2 \left( B_{\vec{\pi}}(n,1/2) \right) > C n^{k+2}.
  \end{align*}
\end{lemma}

\begin{proof} 
Let $x_1, \dots, x_{k-2}, y_0, y_1, z_0, z_1$ be distinct vertices and recall that the vertex set of
$B_{\vec{\pi}}(n,1/2)$ is $[n]$.  There are several cases depending on how the vertices
$x_1,\dots,x_{k-2}, y_0, y_1, \linebreak[1] z_0, z_1$ are ordered in $[n]$.  Let $\mathcal{O} =
\mathcal{O}[x_1;\dots;x_{k-2};y_0,y_1;z_0,z_1]$ and let $c_1 : \binom{[n]}{k-1} \rightarrow \{0,1\}$
and $c_2 : [n] \rightarrow \{0,1\}$ be the two random colorings used in the definition of
$B_{\vec{\pi}}(n,1/2)$.

\begin{itemize}
  \item Case 1: $z_0$ and $z_1$ appear last.  In this case, $|\mathcal{O} \cap E(H)|$ is always even
    as follows.  If $c_2(z_0) = c_2(z_1)$, then either $x_1\dots x_{k-2}y_0z_0$ and $x_1\dots
    x_{k-2}y_0z_1$ are both hyperedges of $B_{\vec{\pi}}(n,1/2)$ or neither are hyperedges depending
    on the value of $c_1(x_1\dots x_{k-2}y_0)$.  Similarly, either $x_1\dots x_{k-2}y_1z_0$ and
    $x_1\dots x_{k-2}y_1z_1$ are both hyperedges or neither are hyperedges so the total number of
    hyperedges induced by $\mathcal{O}$ is even.  If $c_2(z_0) \neq c_2(z_1)$, then exactly one of
    $x_1\dots x_{k-2}y_0z_0$ and $x_1\dots x_{k-2}y_0z_1$ is a hyperedge and exactly one of $x_1\dots
    x_{k-2}y_1z_0$ and $x_1\dots x_{k-2}y_1z_1$ is a hyperedge.  Thus the total number of hyperedges
    induced by $\mathcal{O}$ is even.

  \item Case 2: $y_0$ and $y_1$ appear last.  This case is symmetric to Case 1: the total number of
    hyperedges induced by $\mathcal{O}$ is even.

  \item Case 3: Some $x_i$ appears after $y_0$ and $z_0$.  In this case, the probability that
    $|\mathcal{O} \cap E(H)|$ is even is $\frac{1}{2}$.  Assume that $x_i$ is the largest vertex
    among $x_1,\dots,x_{k-2}$.  The set $x_1\dots x_{k-2}y_0z_0$ is a hyperedge of
    $B_{\vec{\pi}}(n,1/2)$ if $c_1(x_1,\dots,x_{i-1},x_{i+1},\dots,x_{k-2},y_0,z_0) + c_2(x_i) = 0
    \pmod 2$.  Also, $x_1\dots x_{k-2}y_0z_0$ is the only hyperedge of $\mathcal{O}$ which tests the
    value of $c_1(x_1,\dots,x_{i-1},x_{i+1},\dots,x_{k-2},y_0,z_0)$, since this is the only
    hyperedge of $\mathcal{O}$ which includes both $y_0$ and $z_0$.  Therefore, conditioning on the
    other hyperedges of $\mathcal{O}$ and also conditioning on $c_2(x_i)$, with probability
    $\frac{1}{2}$, $c_1(x_1,\dots,x_{i-1},x_{i+1},\dots,x_{k-2},\linebreak[1]y_0,z_0) = 0$ so with
    probability $\frac{1}{2}$, $x_1\dots x_{k-2}y_0z_0$ is a hyperedge, so with probability
    $\frac{1}{2}$, $|\mathcal{O} \cap E(H)|$ is even.

  \item Cases 4-6: Some $x_i$ appears after $y_0, z_1$, some $x_i$ appears after $y_1, z_0$, and
    some $x_i$ appears after $y_1,z_1$.  These three cases are symmetric to Case 3: the probability
    that $|\mathcal{O} \cap E(H)|$ is even is $\frac{1}{2}$.
\end{itemize}

Now consider the sum $\dev_{2}(H_n)$:
\begin{align*}
  \dev_2(H_n) =
  &\sum_{x_1,\dots,x_{k-2},y_0,y_1,z_0,z_1\in \text{Cases 1,2}}
    \eta(x_1;\dots;x_{k-2};y_0,y_1;z_0,z_1) \\
  +
  &\sum_{x_1,\dots,x_{k-2},y_0,y_1,z_0,z_1\in \text{Cases 3-6}}
    \eta(x_1;\dots;x_{k-2};y_0,y_1;z_0,z_1)
\end{align*}
In the sum over Cases 1 and 2, $\eta$ is always $+1$ so the sum is at least $c n^{k+2}$, where $c$
is the fraction of octahedrons in Cases 1 and 2.  Dividing the vertices in half, choosing $z_0,z_1$
from the second half and all other vertices from the first half is a lower bound on $c$, so $c >
2^{-k-3}$.  The expected value of the sum over Cases 3-6 is zero by
linearity of expectation.  Since two octahedrons will depend on each other only if they share a
vertex, the second moment method implies that with high probability the sum over Cases 3-6 is at
most $\frac{c}{2} n^{k+2}$ in absolute value.  Thus with high probability, $\dev_2(H_n) >
\frac{c}{2} n^{k+2}$.
\end{proof} 

\begin{lemma} \label{lem:constrDfailsdev}
  ($D(n,1/2)$ fails \propdev{$2$}). Fix $k \geq 3$.  There exists a constant $C > 0$ such that, with probability
  going to one as $n$ goes to infinity,
  \begin{align*}
    \dev_2 \left( D(n,1/2) \right) > C n^{k+2}.
  \end{align*}
\end{lemma}

\begin{proof} 
This proof is very similar to the proof of Lemma~\ref{lem:constrBfailsdev}.  
Let $x_1, \dots, x_{k-2}, y_0, y_1, z_0, z_1$ be distinct vertices and let $G = G^{(k-1)}(n,1/2)$ be the
random hypergraph used in the definition of $D(n,1/2)$.

\begin{itemize}
  \item Case 1: $\{x_1,\dots,x_{k-2}\}$ is the head of every $k$-tuple in $\mathcal{O}$.  In this
    case, $|\mathcal{O} \cap E(H)|$ is always even.  Indeed, let $\vec{x} = x_1\dots x_{k-2}$ and
    consider the tuples of $\mathcal{O}$ ordered as $\vec{x} y_0z_0$, $\vec{x}y_0z_1$, $\vec{x} y_1
    z_1$, and $\vec{x} y_1 z_0$. (In a drawing of the octahedron with $\vec{x}$ at the center and
    $y_0, z_0, y_1, z_1$ as the corners of a box around the center, these are the edges ordered
    cyclically.) These tuples will be hyperedges depending on if $\vec{x}y_0$, $\vec{x}z_1$,
    $\vec{x}y_1$, and $\vec{x}z_0$ are edges of $G$ or not.  Considering these tuples in this order,
    each transition between edge and non-edge of $G$ implies a missing hyperedge of $\mathcal{O}$
    and each transition between edge and edge or between non-edge and non-edge of $G$ implies a
    hyperedge of $\mathcal{O}$.  Since there are an even number of transitions, $|\mathcal{O} \cap
    E(H)|$ is always even.

  \item Case 2: $\{x_1,\dots,x_{k-2}\}$ is not the head of some tuple in $\mathcal{O}$.  In this
    case, $|\mathcal{O} \cap E(H)|$ is even with probability $\frac{1}{2}$.  Assume by symmetry that
    $y_0$ is included in the head of $x_1\dots x_{k-2}y_0z_0$.  Then $x_1\dots x_{k-2}y_0z_0$ is a
    hyperedge depending on if two $(k-1)$-sets are in $G$, and at least one of these $(k-1)$-sets
    include both $y_0$ and $z_0$.  This $(k-1)$-set including both $y_0$ and $z_0$ is only tested as
    part of deciding if $x_1\dots x_{k-2}y_0z_0$ is a hyperedge, since this is the only tuple of
    $\mathcal{O}$ which includes both $y_0$ and $z_0$.  Thus conditioning on all other tuples of
    $\mathcal{O}$, $x_1\dots x_{k-2}y_0z_0$ is a hyperedge with probability $\frac{1}{2}$ so the
    number of tuples of $\mathcal{O}$ which are hyperedges is even with probability $\frac{1}{2}$.
\end{itemize}

Similar to the proof of Lemma~\ref{lem:constrBfailsdev}, divide the sum $\dev_2(H_n)$ into two sums
by case.  The sum over Case 1 is at least $c n^{k+2}$ for some $c > 0$ and the expected value of the
sum over Case 2 is zero.  Thus using the second moment method, with high probability, $\dev_2(H_n) >
\frac{c}{2} n^{k+2}$.
\end{proof} 

\subsection{Expansion} 

In this section, we show that with high probability $A_2(n,p)$, $B_{\vec{\pi}'}(n,p)$, and
$D(n,\frac{1}{2})$ satisfy \propexpand{$\pi$} if $\pi'$ is not a refinement of $\pi$.  The proof
generalizes to show that $A_\ell(n,p)$ satisfies \propexpand{$\pi$} for all $\ell$, but this is not
required so the proof is omitted. To show these constructions satisfy \propexpand{$\pi$}, we take
advantage of a theorem of the current authors~\cite{hqsi-lenz-quasi12} which shows that two properties
on counting subgraphs are equivalent to \propexpand{$\pi$}.  Counting subgraphs is easier than
showing \propexpand{$\pi$} holds, so using~\cite{hqsi-lenz-quasi12} simplifies the proof.

\begin{definition}
  Let $k \geq 2$ and let $\pi = k_1+\dots+k_t$ be a proper partition of $k$.  A $k$-uniform
  hypergraph $F$ is \emph{$\pi$-linear} if there exists an ordering $E_1,\dots,E_m$ of the edges of
  $F$ such that for every $i$, there exists a partition of the vertices of $E_i$ into
  $A_{i,1},\dots,A_{i,t}$ such that for $1 \leq s \leq t$, $|A_{i,s}|=k_s$ and for every $j < i$, there exists
  an $s$ such that $E_j \cap E_i \subseteq A_{i,s}$. 
\end{definition}

\begin{definition}
  Let $k \geq 2$ and let $\pi = k_1+k_2$ be a partition of $k$ into two parts.  The cycle
  $C_{\pi,4}$ of type $\pi$ and length four is the following hypergraph.  Let $X_1, X_2, Y_1, Y_2$ be
  disjoint sets with $|X_1| = |X_2| = k_1$ and $|Y_1| = |Y_2| = k_2$.  The vertex set of $C_{\pi,4}$
  is $X_1 \cup X_2 \cup Y_1 \cup Y_2$ and the edge set is $\{ X_i \cup Y_j : 1 \leq i,j
  \leq 2\}$.
\end{definition}

Among other things, the current authors~\cite{hqsi-lenz-quasi12, hqsi-lenz-quasi12-nonregular}
proved that the properties \propcount{$\pi$-linear} and \propcycle{$\pi$} (defined below) are
equivalent to \propexpand{$\pi$}.  If $F$ and $H$ are hypergraphs, a \emph{labeled copy of $F$ in
$H$} is an edge-preserving injection $V(F) \rightarrow V(H)$, i.e.\ an injection $\alpha : V(F)
\rightarrow V(H)$ such that if $E$ is an edge of $F$, then $\{ \alpha(x) : x \in E \}$ is an edge of
$H$.

\begin{thm} \label{thm:piquasiiscounting} (\cite{hqsi-lenz-quasi12, hqsi-lenz-quasi12-nonregular})
  Let $\mathcal{H} = \{H_n\}_{n\rightarrow\infty}$ be a sequence of $k$-uniform hypergraphs where $|V(H_n)| = n$
  and $|E(H_n)| \geq p \binom{n}{k} + o(n^k)$.  Let $\pi$ be any proper partition of $k$.  Then
  $\mathcal{H}$ satisfies \propexpandp{$\pi$} if and only if $\mathcal{H}$ satisfies
  \begin{itemize}
     \item \propcountp{$\pi$-linear}: For all $f$-vertex, $m$-edge, $k$-uniform, $\pi$-linear
       hypergraphs $F$, the number of labeled copies of $F$ in $H_n$ is $p^m n^f + o(n^f)$.
  \end{itemize}
  In addition, if $\pi = k_1+k_2$ is a partition into two parts, $\mathcal{H}$ satisfies
  \propexpandp{$\pi$} if and only if $\mathcal{H}$ satisfies
  \begin{itemize}
    \item  \propcyclep{$\pi$}: The number of labeled copies of $C_{\pi,4}$ in $H_n$ is at most
      $p^{4} n^{2k} + o(n^{2k})$.
  \end{itemize}
\end{thm}

Note that \cite{hqsi-lenz-quasi12,hqsi-lenz-quasi12-nonregular} actually defines a cycle
$C_{\pi,2\ell}$ for any proper partition $\pi$ and any $\ell \geq 2$ and equates counting cycles
with \propexpand{$\pi$}, but the full definition of $C_{\pi,2\ell}$ is complicated and not required
in this paper.  Therefore, we only state the definition of cycles and the equivalence between
counting cycles and expansion for partitions into two parts.

\begin{lemma} \label{lem:constrAsatisfiesexp}
  ($A_2(n,p)$ satisfies \propcycle{$k_1,k_2$}) For $k = k_1 + k_2$ with $k_i \geq 1$, $\epsilon >
  0$, and $0 < p < 1$, with probability going to one as $n$ goes to infinity, the number of labeled
  copies of $C_{k_1+k_2,4}$ in $A_2(n,p)$ satisfies
  \begin{align} \label{eq:constrAsatCD}
    \left| \#\{C_{k_1+k_2,4} \,\, \text{in } A_2(n,p) \} - p^4 n^{2k} \right| < \epsilon
    n^{2k}.
  \end{align}
\end{lemma}

\begin{proof} 
Let $c : E(K_n) \rightarrow \{0,\dots,b-1\}$ be the random coloring used in the construction of
$A_2(n,p)$.  The cycle $C_{\pi,4}$ has four edges with four vertex groups $X_1,X_2,Y_1,Y_2$ where
$|X_i| = k_1$ and $|Y_i|=k_2$ for all $i$ and $X_i \cup Y_j$ are hyperedges for all $i,j$.  Let us
pick disjoint sets $X_1,X_2,Y_1,Y_2$ of vertices of $A_2(n,p)$ and compute the probability that each
$X_i \cup Y_j$ is a hyperedge of $A_2(n,p)$.  We claim that the probability that $X_i \cup Y_j$ is a
hyperedge of $A_2(n,p)$ is $p$ independently of if the other pairs are hyperedges or not.  The only
possible dependence between the events ``$X_i \cup Y_j$ is a hyperedge of $A_2(n,p)$'' and the event
``$X_{i'} \cup Y_{j'}$ is a hyperedge of $A_2(n,p)$'' would come from the edges of $K_n$ appearing
in the intersection of the two hyperedges.  But even conditioned on exactly the behavior of the
colors of $\binom{X_i}{2}$ and the colors of $\binom{Y_i}{2}$, $X_i \cup Y_j$ is an edge with
probability $p$.  This is because if $x_iy_j$ is a pair of vertices with one endpoint in $X_i$ and
one endpoint in $Y_j$, even conditioning on the color of every pair in $\binom{X_i \cup Y_j}{2}$
besides $\{x_i,y_j\}$, the set $X_i \cup Y_j$ is a hyperedge of $A_2(n,p)$ or not depending only on
the color of $x_iy_j$ (similar to the proof of Lemma~\ref{lem:Aispdense}).

Thus since each $X_i \cup Y_j$ is a hyperedge of $A_2(n,p)$ with probability
$p$, the expected number of labeled cycles $C_{\pi,4}$ in $A_2(n,p)$ is $p^{4} (2k)!
\binom{n}{2k}$.  Since two cycles will depend on each other only if they
share at least one vertex, the second moment method implies that with high probability, the number
of labeled cycles in $A_2(n,p)$ is $p^4 n^{2k} \pm \epsilon n^{2k}$.
\end{proof} 

The above proof generalizes in a straightforward manner to show that $A_\ell(n,p)$ satisfies
\propcycle{$k_1 + k_2$}, although we do not require this fact in this paper.  Also, since every
partition $\pi$ is a refinement of $k_1 + k_2$ for some $k_1$ and $k_2$, $A_\ell(n,p)$ satisfies
\propexpand{$\pi$} for all $\pi$.

\begin{lemma} \label{lem:constrBsatisfiesexp}
  ($B_{\vec{\pi}'}(n,p)$ satisfies \propcount{$\pi$-linear})  Let $\pi$ and $\pi'$ be proper
  partitions of $k$ such that $\pi'$ is not a refinement of $\pi$.  Let $\vec{\pi}'$ be any ordering
  of the entries of $\pi'$.  For any $\pi$-linear, $v$-vertex, $m$-edge hypergraph $F$, any
  $\epsilon > 0$, and any $0 < p < 1$, with probability going to one as $n$ goes to infinity,
  \begin{align*}
    \left| \#\{F \, \text{in } B_{\vec{\pi}'}(n,p)\} - p^{m} n^v \right| < \epsilon n^v.
  \end{align*}
\end{lemma}

\begin{proof} 
Let $F$ be a $\pi$-linear hypergraph with $v$ vertices, labeled $f_1,\dots,f_v$,
and let $x_1,\dots,x_v \in V(B_{\vec{\pi}'}(n,p))$ be a list of $v$
distinct vertices of $B_{\vec{\pi}'}(n,p)$. Define an indicator random variable
\begin{align*}
  X(F;x_1,\dots,x_v) = \begin{cases}
    1, &\text{if } B_{\vec{\pi}'}[\{x_1,\dots,x_v\}] \, \text{is a labeled copy of $F$ with $x_i$ mapped to $f_i$}, \\
    0  &\text{otherwise}.
  \end{cases}
\end{align*}

\medskip

\noindent\textit{Claim: } For any $\pi$-linear hypergraph $F$, and any distinct vertices $x_1,\dots,x_v \in
V(B_{\vec{\pi}}(n,p))$, we have $\mathbb{P}[X(F;x_1,\dots,x_{v}) = 1] = p^{|E(F)|}$.  

\medskip

\begin{proof} 
The claim is proved by induction on the number of edges of $F$.  If $F$ has no edges, then
$X(F;x_1,\dots,x_v)$ is always one.  For the inductive step, let $E$ be the last edge of $F$ in the
ordering provided by the $\pi$-linearity of $F$, and let $m = |E(F)|$.  We may assume that the
vertices of $F$ are labeled so that $E = \{f_1,\dots,f_k\}$.  Let $F'$ be the hypergraph with $V(F')
= V(F)$ and $E(F') = E(F) - E$.  Then
\begin{align} \label{eq:constrBcountinglinear}
  \mathbb{P}[X(F;x_1,\dots,x_v) = 1] 
  = &\mathbb{P}[\{x_1,\dots,x_k\} \in E(B_{\vec{\pi}'}(n,p)) \Big| X(F';x_1,\dots,x_v) = 1] \nonumber \\
       \cdot &\mathbb{P}[X(F';x_1,\dots,x_v) = 1].
\end{align}
By induction, $\mathbb{P}[X(F';x_1,\dots,x_v) = 1]$ is $p^{m-1}$ so let us investigate the
probability that $\{x_1,\dots,x_k\}$ forms an edge of $B_{\vec{\pi}'}(n,p)$ conditioned on
$x_1,\dots,x_v$ forming a copy of $F'$.  The way we test if $\{x_1,\dots,x_k\}$ forms an edge of
$B_{\vec{\pi}'}(n,p)$ is to sort the $x_i$s according to the underlying ordering of
$V(B_{\vec{\pi}'}(n,p))$ and test the color sum of the $\pi'$-groups.  More precisely, let $\eta$ be
the permutation of $[k]$ such that $x_{\eta(1)} < x_{\eta(2)} < \dots < x_{\eta(k)}$ and let
$\vec{\pi}' = (k'_1,\dots,k'_t)$.  Divide the $x_{\eta(i)}$s up into blocks $D_1,\dots,D_t$ so that
$D_1$ consists of $x_{\eta(1)}, \dots, x_{\eta(k'_1)}$, $D_2$ consists of the next $k'_2$ vertices,
and so on.  The set $\{x_1,\dots,x_k\}$ will be an edge of $B_{\vec{\pi}'}(n)$ if $\sum c_i(D_i) < a
\pmod b$.

Since $\pi'$ is not a refinement of $\pi$ and $F$ is $\pi$-linear, there is some block $D_i$ such
that no edge of the copy of $F'$ on $x_1,\dots,x_v$ completely contains the block $D_i$.  To see
this, assume for contradiction that every block is completely contained in some edge of $F'$.
Since $F$ is $\pi$-linear, there exists a partition of the vertices of $E$ into groups according to
the partition $\pi$ such that every edge of $F'$ intersects at most one of these parts.  If every
block $D_i$ (which came from the $\pi'$ partition) was completely contained inside some edge, it
would be completely contained inside the corresponding part of the $\pi$-partition of $E$.  This
assignment of blocks to parts of the $\pi$-partition of $E$ shows that $\pi'$ is a refinement of
$\pi$, which is a contradiction.

Thus there exists some block $D_i$ which is not contained inside any edge of the copy of $F'$ on
$x_1,\dots,x_v$, so the event ``$D_i \in E(G_i)$'' is independent of the event
``$X(F',x_1,\dots,x_v) = 1$''.  Moreover, the event ``$\{x_1,\dots,x_k\}$ forms an edge of
$B_{\vec{\pi}'}(n,p)$'' can be written in terms of the event ``$D_i \in E(G_i)$,'' since no matter what
happens to $D_j$ for $j \neq i$, $D_i$ has probability $p = \frac{a}{b}$ of making the total color
sum in $\{0,\dots,a-1\}$ (similar to the proof of Lemma~\ref{lem:Bispdense}).  Combining this with
\eqref{eq:constrBcountinglinear} and induction on the number of edges finishes proves the claim.
\end{proof} 

By linearity of expectation, the expected number of labeled copies of $F$ in $B_{\vec{\pi}'}(n,p)$
is $p^m v! \binom{n}{v}$.  Since two events will depend on each other only if the copies
of $F$ share at least two vertices, the second moment method implies that with high probability, the
number of labeled copies of $F$ is $p^m n^v \pm \epsilon n^v$.
\end{proof} 

\begin{lemma} \label{lem:constrDsatsfiesexp}
  ($D(n,\frac{1}{2})$ satisfies \propcycle{$k_1 + k_2$}).  For $k = k_1 + k_2$ with $k_i \geq 1$,
  $\epsilon > 0$, and $0 < p < 1$, with probability going to one as $n$ goes to infinity, the number
  of labeled copies of $C_{k_1+k_2,4}$ in $D(n,\frac{1}{2})$ satisfies
  \begin{align} \label{eq:constrCsatCycle}
    \left| \#\left\{C_{k_1+k_2,4} \,\, \text{in } D\left(n,1/2\right) \right\} -
    \left(1/2\right)^4 n^{2k} \right| < \epsilon n^{2k}.
  \end{align}
\end{lemma}

\begin{proof} 
This proof is very similar to the proof of Lemma~\ref{lem:constrAsatisfiesexp}; we will show that
the probability that some $2k$ vertices form a labeled copy of $C_{\pi,4}$ in $D(n,1/2)$ is $2^{-4}$
and use the second moment method for concentration.  Let $G = G^{(k-1)}(n,1/2)$ be the random
hypergraph used in the definition of $D(n,1/2)$.  Let $X_1,X_2,Y_1,Y_2$ be disjoint sets of vertices
of $D(n,1/2)$ with $|X_i| = k_1$ and $|Y_i| = k_2$.  We claim that the probability that $X_i \cup
Y_j$ is a hyperedge of $D(n,1/2)$ is $\frac{1}{2}$ independently of if the other pairs are
hyperedges or not.  Indeed, let $R$ be the head of $X_i \cup Y_j$ and $z_1$ and $z_2$ the other
two vertices of $X_i \cup Y_j$, and notice that since $|R| = k-2$ either $R \cup \{z_1\}$ or $R \cup
\{z_2\}$ (or both) intersect both $X_i$ and $Y_j$.  Say $R \cup \{z_1\}$ intersects both $X_i$ and
$Y_j$.  Since the sets $X_1, X_2, Y_1, Y_2$ are disjoint, $X_i \cup Y_j$ is the only hyperedge
of the cycle to test if $R \cup \{z_1\}$ is in $G$ or not.  Since $R \cup \{z_1\}$ is an edge of $G$
with probability $\frac{1}{2}$, $X_i \cup Y_j$ is a hyperedge of $D(n,1/2)$ with probability
$\frac{1}{2}$ independently of the other hyperedges of the cycle.

Thus the expected number of labeled four-cycles in $D(n,1/2)$ is $2^{-4} (2k)!  \binom{n}{2k}$ and by
the second moment method, with probability going to one as $n$ goes to infinity, the number of
labeled copies of $C_{\pi,4}$ in $D(n,1/2)$ is $2^{-4} n^{2k} \pm \epsilon n^{2k}$.
\end{proof} 

\subsection{Clique Discrepency} 

In this section, we show that $A_{\ell+1}(n,p)$ and $B_{\vec{\pi}}(n,p)$ satisfy \propcd{$\ell$} for
$\vec{\pi} = (\ell+1,1,\dots,1)$.  The proof generalizes to show that $B_{\vec{\pi}}(n,p)$ satisfies
\propcd{$\ell$} for all $\pi$ with $\max \pi > \ell$, but this generalization is not required for
Table~\ref{tab:main} so the proof is omitted.

\begin{lemma} \label{lem:constrAsatCD}
  ($A_{\ell+1}(n,p)$ satisfies \propcd{$\ell$})  Let $1 \leq \ell < k-1$ and $0 < p < 1$.  For every
  $\epsilon > 0$, with probability going to one as $n$ goes to infinity, for every $\ell$-uniform
  hypergraph $F$ on vertex set $[n]$,
  \begin{align} \label{eq:constrAsatCDmain}
    \Big| |\mathcal{K}_k(F) \cap E(A_{\ell+1}(n,p))| - p|\mathcal{K}_k(F)| \Big| \leq \epsilon n^k.
  \end{align}
\end{lemma}

\begin{proof} 
Let $A = A_{\ell+1}(n,p)$ and let $c : \binom{V(A)}{\ell+1} \rightarrow \{0,\dots,b-1\}$ be the random coloring used
in the definition of $A$.  Fix some $\ell$-uniform hypergraph $F$ on vertex set
$V(A)$ and let us compute the probability that $F$ is bad, where \emph{bad} means that
\eqref{eq:constrAsatCDmain} fails.  Let $x_1,\dots,x_{k-\ell-1}$ be distinct vertices and define
\begin{align*}
  W_{x_1,\dots,x_{k-\ell-1}} &= \left\{ w \in \binom{V(A)}{\ell+1} : w \cup
  \{x_1,\dots,x_{k-\ell-1}\} \in \mathcal{K}_k(F) \right\}, \\
  Y_{x_1,\dots,x_{k-\ell-1}} &= \left\{ w \cup \{x_1,\dots,x_{k-\ell-1}\} : w \in W_{x_1,\dots,x_{k-\ell-1}}
  \right\}.
\end{align*}
Call $x_1,\dots,x_{k-\ell-1}$ \emph{bad} if
\begin{align} \label{eq:constrAxbad}
  \Bigg| |Y_{x_1,\dots,x_{k-\ell-1}} \cap E(A)| - p |Y_{x_1,\dots,x_{k-\ell-1}}|
  \Bigg| > \epsilon n^{\ell+1}.
\end{align}
If $F$ is bad, then some $x_1,\dots,x_{k-\ell-1}$ is bad.  Indeed, if every $x_1,\dots,x_{k-\ell-1}$
was good, then using that $e < 2.8$ we have that
\begin{align*}
  \Bigg| |\mathcal{K}_k(F) &\cap E(A)| - p|\mathcal{K}_k(F)| \Bigg| \\
  &= \binom{k}{k-\ell-1}^{-1}
  \Bigg| \sum_{\{x_1,\dots,x_{k-\ell-1}\} \in \binom{V(A)}{k-\ell-1}} 
      \Big( |Y_{x_1,\dots,x_{k-\ell-1}} \cap E(A)| - p |Y_{x_1,\dots,x_{k-\ell-1}}| \Big) \Bigg| \\
  &\leq \binom{k}{k-\ell-1}^{-1}
    \sum_{\{x_1,\dots,x_{k-\ell-1}\} \in \binom{V(A)}{k-\ell-1}}
      \Big| |Y_{x_1,\dots,x_{k-\ell-1}} \cap E(A)| - p |Y_{x_1,\dots,x_{k-\ell-1}}| \Big| \\
  &\leq \binom{k}{k-\ell-1}^{-1}
    \sum_{\{x_1,\dots,x_{k-\ell-1}\} \in \binom{V(A)}{k-\ell-1}} \epsilon n^{\ell+1} \\
  &= \frac{\binom{n}{k-\ell-1}}{\binom{k}{k-\ell-1}} \, \epsilon n^{\ell+1} \\
  &\leq \frac{\left(\frac{2.8 n}{k-\ell-1}\right)^{k-\ell-1}}{\left( \frac{k}{k-\ell-1}
    \right)^{k-\ell-1}} \, \epsilon n^{\ell+1} \\
  &= \left( \frac{2.8}{k} \right)^{k-\ell-1} \, \epsilon n^k.
\end{align*}
Since $1 \leq \ell < k - 1$, we have $k \geq 3$ so that $\left(\frac{2.8}{k}\right)^{k+\ell-1} \leq
1$.  Thus if every $x_1,\dots,x_{k-\ell-1}$ was good, $F$ would be good.  Therefore, using the union
bound, the probability that $F$ is bad is $\binom{n}{k-\ell-1}$ times the probability that
$x_1,\dots,x_{k-\ell-1}$ is bad.

Let us compute the probability that $x_1,\dots,x_{k-\ell-1}$ is bad.  For each $w \in
W_{x_1,\dots,x_{k-\ell-1}}$, define $X_w$ as the following indicator random variable:
\begin{align*}
  X_w = \begin{cases}
    1 \, &\text{if } \displaystyle\sum_{r \in \binom{w \cup \{x_1,\dots,x_{k-\ell-1}\}}{\ell+1}}
    c(r) < a \pmod b \\
    0 &\text{otherwise}.
  \end{cases}
\end{align*}
Notice that the event ``$X_w = 1$'' is mutually independent of the events ``$X_{w'} = 1$'' for $w'
\neq w$, since $X_w$ is the only event to test the color assigned to $w$.  Indeed, conditioning on
all other $r \in \binom{w\cup\{x_1,\dots,x_{k-\ell-1}\}}{\ell+1}$ with $r \neq w$, the probability
that $X_w = 1$ is $p$.  Let $X = \sum X_w$ and $\mu = \mathbb{E}[X]$ and notice that $X =
|Y_{x_1,\dots,x_{k-\ell-1}} \cap E(A)|$ and $\mu = p|W_{x_1,\dots,x_{k-\ell-1}}| =
p|Y_{x_1,\dots,x_{k-\ell-1}}|$.  Thus \eqref{eq:constrAxbad} becomes $|X - \mu| > \epsilon
n^{\ell+1}$.

Next, by Chernoff's Bound (Lemma~\ref{lem:chernoff}), there exists a constant $c$ depending only on
$\epsilon$ such that
\begin{align} \label{eq:AsatCDchernofftarget}
  \mathbb{P}\left[| X - \mu| > \epsilon n^{\ell+1} \right] \leq e^{-c n^{\ell+1}}.
\end{align}
Indeed, let $a = \epsilon n^{\ell+1}$ and let $n' = |W_{x_1,\dots,x_{k-\ell-1}}| \leq n^{\ell+1}$ be
the number of indicator random variables.  Then by Chernoff's Bound, $\mathbb{P}[|X - \mu| > a] \leq
2e^{-a^2/2n'} \leq 2e^{-\epsilon^2 n^{\ell+1}/2} \leq e^{-cn^{\ell+1}}$ if $c = \epsilon^2/4$.

If $F$ is bad, then some $x_1,\dots,x_{k-\ell-1}$ is bad, so the union bound implies that
the probability that $F$ is bad is at most
\begin{align*}
  \binom{n}{k-\ell-1} e^{-c n^{\ell+1}} \leq e^{-c n^{\ell+1}/2}.
\end{align*}
Apply the union bound again to compute the probability that some $F$ is bad.  There are at most
$2^{n^{\ell}}$ choices for $F$ so the probability that some $F$ is bad is at most
\begin{align*}
  2^{n^{\ell}} e^{-cn^{\ell+1}/2} \leq e^{-c n^{\ell+1}/4}
\end{align*}
which goes to zero as $n$ goes to infinity, completing the proof.
\end{proof} 

\begin{lemma} \label{lem:constrBsatisfiesCD}
  ($B_{\ell+1,1,\dots,1}(n,p)$ satisfies \propcd{$\ell$}) Let $1 \leq \ell < k-1$, $0 < p < 1$, and
  let $\vec{\pi} = (\ell+1,1,\dots,1)$. For every $\epsilon > 0$, with probability going to one as $n$
  goes to infinity, for every $\ell$-uniform hypergraph $F$ on vertex set $[n]$, 
  \begin{align} \label{eq:constrBcorrectcliques}
    \Big| |\mathcal{K}_k(F) \cap B_{\vec{\pi}}(n,p)| - p|\mathcal{K}_k(F)| \Big| \leq \epsilon n^k.
  \end{align}
\end{lemma}

\begin{proof} 
This proof is similar to the proof of the previous lemma.
Fix some $\ell$-uniform hypergraph $F$ on vertex set $V(B_{\vec{\pi}}(n,p)) = [n]$ and let us compute the
probability that $F$ is bad, where \emph{bad} means that \eqref{eq:constrBcorrectcliques} fails.
Recall that since $\vec{\pi} = (\ell+1,1,\dots,1)$, $B_{\vec{\pi}}(n,p)$ is built from a random coloring $c_1$ of
the complete $(\ell+1)$-uniform hypergraph and $k-\ell-1$ random colorings $c_2,\dots,c_{k-\ell}$ of
the complete one-uniform hypergraph.

Fix $k-\ell-1$ distinct vertices $x_2,\dots,x_{k-\ell}$ with $x_2 < \dots < x_{k-\ell}$ and let $W$
be the collection of $(\ell+1)$-sets which contain elements earlier than $x_2$ in the ordering and
also form a clique of size $k$ in $F$ when added to $x_2,\dots,x_{k-\ell}$.  More precisely,
\begin{align*}
  W_{x_2,\dots,x_{k-\ell}} = \left\{ w : w \in \binom{[x_2-1]}{\ell+1},  w \cup \{x_2,\dots,x_{k-\ell}\} \in 
    \mathcal{K}_k(F)\right\}.
\end{align*}
Notice that we define $W_{x_2,\dots,x_{k-\ell}}$ as $(\ell+1)$-sets of elements smaller than $x_2$,
so that asking if $w \cup \{x_2,\dots,x_{k-\ell}\}$ is an edge of $B_{\vec{\pi}}(n,p)$ consists of
asking about the color of $w$ in $c_1$ and the colors of $x_2,\dots,x_{k-\ell}$ in
$c_2,\dots,c_{k-\ell}$.  Since $x_2,\dots,x_{k-\ell}$ are fixed, define $\Delta =
\sum_{j=2}^{k-\ell} c_j(x_j)$.  For each $w \in W_{x_2,\dots,x_{k-\ell}}$, define a random variable
$X_w$ as follows.
\begin{align*}
  X_w = \begin{cases}
    1 &\text{if } c_1(w) + \Delta < a \pmod b, \\
    0 &\text{otherwise}.
  \end{cases}
\end{align*}
Since $\Delta$ is fixed, the expectation $\mathbb{E}[X_w] = \frac{a}{b} = p$.  Also, all these
indicator random variables are mutually independent.  Define $\hat{G}_\Delta$ to be the
$(\ell+1)$-uniform hypergraph on vertex set $V(B_{\vec{\pi}}(n,p))$ whose hyperedges are the
$(\ell+1)$-sets receiving colors $\{-\Delta,-\Delta+1,\dots,-\Delta+a-1\} \pmod b$.  Define $X =
\sum X_w$ and $\mu = \mathbb{E}[X] = p|W_{x_2,\dots,x_{k-\ell}}|$.  Consider an $(\ell+1)$-set $w$
in $E(\hat{G}_\Delta) \cap W_{x_2,\dots,x_{k-\ell}}$.  Then the color sum of $w \cup
\{x_2,\dots,x_{k-\ell}\}$ is between $0$ and $a-1 \pmod b$ so that $w \cup \{x_2,\dots,x_{k-\ell}\}$
is a hyperedge of $B_{\vec{\pi}}(n,p)$ and $X_w = 1$.  In the other direction, if $X_w = 1$ then the
color sum of $w \cup \{x_2,\dots,x_{k-\ell}\}$ is between $0$ and $a-1 \pmod b$ which implies that
$w \in E(\hat{G}_\Delta)$.  Therefore, $X = |E(\hat{G}_\Delta) \cap W_{x_2,\dots,x_{k-\ell}}|$.

By a similar argument as in the proof of Lemma~\ref{lem:constrAsatCD}, the Chernoff Bound
(Lemma~\ref{lem:chernoff}) implies that there exists a constant $c > 0$ such that
\begin{align} \label{eq:constrBsatCDchernoff}
  \mathbb{P}\left[| X - \mu| > \epsilon n^{\ell+1} \right] < e^{-c n^{\ell+1}}.
\end{align}
Call $x_2,\dots,x_{k-\ell}$ \emph{bad} if
\begin{align*}
  \left| X - \mu \right| = \Bigg| \left| E(\hat{G}_\Delta) \cap W_{x_2,\dots,x_{k-\ell}} \right| - 
  p \left| W_{x_2,\dots,x_{k-\ell}} \right| \Bigg|
  > \epsilon n^{\ell+1}.
\end{align*}

Next, we claim that if $F$ is bad then there exists some choice of $x_2,\dots,x_{k-\ell}$ which is
bad.  To see this, notice that the $k$-cliques of $F$ can be partitioned based on their $k-\ell-1$
largest vertices.  Let $Y_{x_2,\dots,x_{k-\ell}} = \{w \cup \{x_2,\dots,x_{k-\ell}\} : w \in
W_{x_2,\dots,x_{k-\ell}}\}$ so that $\mathcal{K}_k(F) = \dot\cup Y_{x_2,\dots,x_{k-\ell}}$.
There are at most $\binom{n}{k-\ell-1} \leq n^{k-\ell-1}$ sets $Y_{x_2,\dots,x_{k-\ell}}$ and they are all
disjoint, so if $F$ is bad then there is some $x_2,\dots,x_{k-\ell}$ such that
\begin{align} \label{eq:constrBsatCDYbad}
  \Big| |E(B_{\vec{\pi}}(n,p)) \cap Y_{x_2,\dots,x_{k-\ell}}| - p|Y_{x_2,\dots,x_{k-\ell}}| \Big|
  > \epsilon n^{\ell+1}.
\end{align}
But $|E(B_{\vec{\pi}}(n,p)) \cap Y_{x_2,\dots,x_{k-\ell}}| = | E(\hat{G}_\Delta) \cap
W_{x_2,\dots,x_{k-\ell}} |$ and $|Y_{x_2,\dots,x_{k-\ell}}| = |W_{x_2,\dots,x_{k-\ell}}|$, so that
\eqref{eq:constrBsatCDYbad} implies that $x_2,\dots,x_{k-\ell}$ is bad.

By the union bound, the probability that $F$ is bad is the number of choices for
$x_2,\dots,x_{k-\ell}$ times the probability that $x_2,\dots,x_{k-\ell}$ is bad, which is bounded by
\eqref{eq:constrBsatCDchernoff}.  Thus the probability that $F$ is bad is at most
\begin{align*}
  \binom{n}{k-\ell-1} e^{-c n^{\ell+1}} \leq e^{-c n^{\ell+1}/2}.
\end{align*}
We apply the union bound again to compute the probability that some $F$ is bad.  There are at most
$2^{n^\ell}$ choices for $F$ so the probability that some $F$ is bad is at most
\begin{align*}
  2^{n^\ell} e^{-c n^{\ell+1}/2} \leq e^{-cn^{\ell+1}/4}
\end{align*}
which goes to zero as $n$ goes to infinity, completing the proof.
\end{proof} 

While not required in this paper, the above proof generalizes to prove that $B_{\vec{\pi}}(n,p)$
satisfies \propcd{$\ell$} for all $\pi$ with $\max \pi > \ell$.  If $\max \pi = k_i$, then the Chernoff
Bound will imply a bound of $e^{-c n^{k_i}}$ which is enough to dominate the term $2^{n^{\ell}}$ from the
number of $\ell$-uniform hypergraphs $F$.

\subsection{Deviation} 
\label{sub:constrsatdev}

\begin{lemma} \label{lem:constrAsatdev}
  ($A_\ell(n,1/2)$ satisfies \propdev{$\ell$}) For every $\epsilon > 0$, with probability going to
  one as $n$ goes to infinity,
  \begin{align*}
    \dev_{\ell} \left( A_\ell(n,1/2) \right) \leq \epsilon n^{k+\ell}.
  \end{align*}
\end{lemma}

\begin{proof} 
This proof is similar to the proofs of Lemmas~\ref{lem:constrBfailsdev}
and~\ref{lem:constrDfailsdev} except that in all cases the probability that a squashed octahedron is
even is $\frac{1}{2}$.  Let $x_1, \dots, x_{k-\ell}, y_{1,0}, y_{1,1}, \dots, y_{\ell,0},
y_{\ell,1}$ be distinct vertices and let $c : \binom{V(A_\ell(n,1/2))}{\ell} \rightarrow \{0,1\}$ be
the random coloring used in the definition of $A_\ell(n,1/2)$.  Let $G$ be the $\ell$-uniform
hypergraph whose hyperedges are those $\ell$-sets colored one.  Note that by definition, a set $T$
of $k$ vertices is a hyperedge of $A_\ell(n,1/2)$ if $|E(G[T])|$ is even.

Let $\mathcal{O} = \mathcal{O}[x_1;\dots;x_{k-\ell};y_{1,0},y_{1,1};\dots;y_{\ell,0},y_{\ell,1}]$.
We will show that with probability $\frac{1}{2}$, $|\mathcal{O} \cap E(A_\ell(n,1/2))|$ is even.
Consider the tuple $ (x_1,\dots,x_{k-\ell},y_{1,0},\dots,y_{\ell,0}) \in \mathcal{O}$ and let $Y = \{
y_{1,0}, \dots, y_{\ell,0}\}$. Note that $\{x_1,\dots, x_{k-\ell}, y_{1,0},\dots, y_{\ell,0}\}$ is a
hyperedge of $A_\ell(n,1/2)$ if the number of edges of $G$ induced by $\{x_1, \dots, x_{k-\ell},
y_{1,0}, \dots, y_{\ell,0}\}$ is even.  But this tuple is the only tuple of $\mathcal{O}$ to test if
$Y$ is an edge of $G$ or not, since no other tuple in $\mathcal{O}$ contains $Y$.  Thus conditioning
on all other tuples in $\mathcal{O}$ and all other $\ell$-subsets of $\{x_1, \dots, x_{k-\ell},
y_{1,0}, \dots, y_{\ell,0}\}$, $Y \in E(G)$ with probability $\frac{1}{2}$ so $x_1\dots x_{k-\ell}
y_{1,0}\dots y_{\ell,0}$ is a hyperedge of $A_\ell(n,1/2)$ with probability $\frac{1}{2}$ so
$|\mathcal{O} \cap E(A_\ell(n,1/2))|$ is even with probability $\frac{1}{2}$.  The expected value of
the sum $\dev_\ell(A_\ell(n,1/2))$ is zero, and the second moment method shows that, with high
probability, $\dev_\ell(A_\ell(n,1/2)) \leq \epsilon n^{k+\ell}$.
\end{proof} 

\section{Non-implications in Table~\ref{tab:main}} 
\label{sec:notimplications}

This section completes the proof of Table~\ref{tab:main} using the hypergraphs constructed in the
previous section. The proof technique to construct hypergraph sequences satisfying some property and
failing some other property is a diagonalization argument and is similar for all the results in
Table~\ref{tab:main}.  First we use the probabilistic method to prove that for every $\epsilon > 0$
and every $n \geq n_0$, with high probability there exists a hypergraph satisfying some property and
failing another.  We then construct a hypergraph sequence by creating a hypergraph for each
$\epsilon = 1/n$ via diagonalization.  Since all the proofs are very similar, we only give the full
proof of \propexpand{$\pi$} $\not\Rightarrow$ \propexpand{$\pi'$} when $\pi'$ is not a refinement of
$\pi$.

\begin{lemma} \label{lem:satisfypifailpi'}
  Let $0 < p < 1$ with $p \in \mathbb{Q}$ and let $\pi$ and $\pi'$ be proper partitions of $k$ such
  that $\pi'$ is not a refinement of $\pi$.  For every $\epsilon > 0$, there exists a $N_0$ such
  that for $n \geq N_0$ there exists a hypergraph $B$ on $n$ vertices such that
  \begin{itemize}
    \item $\Big| |E(B)| - p\binom{n}{k} \Big| \leq \epsilon n^k$,

    \item If $\pi' = k'_1 + \dots + k'_t$, then there exists $S'_1,\dots,S'_t$ with $S'_i \subseteq
      \binom{V(B)}{k'_i}$ and a constant $C > 0$ depending only on $p$, $k$, $\pi'$, and $t$ such that
      \begin{align*}
        \Big|e(S'_1,\dots,S'_t) - p|S'_1|\dots|S'_t| \Big| > C \binom{n}{k},
      \end{align*}

    \item For every $\pi$-linear hypergraph $F$ with $v$ vertices and $e$ edges where $v \leq
      \epsilon^{-1}$,
      \begin{align*}
        |\#\{F \, \text{in } B\} - p^{e} n^v| < \epsilon n^v.
      \end{align*} 
  \end{itemize}
\end{lemma}

\begin{proof} 
Let $v_0 = \epsilon^{-1}$ and let $\eta = \frac{1}{2} 2^{-2^{v_0}}$.  By Lemma~\ref{lem:Bispdense},
$N_0$ can be chosen large enough so that for $n \geq N_0$, with probability at most $\eta$, $B =
B_{\vec{\pi}'}(n,p)$ has $|E(B) - p\binom{n}{k}| \geq \epsilon n^k$.  There are $2^{2^{v_0}}$
hypergraphs with $v_0$ vertices. Thus, for each of the at most $2^{2^{v_0}}$ $\pi$-linear
hypergraphs $F$ with at most $v_0$ vertices, Lemma~\ref{lem:constrBsatisfiesexp} shows that we can
choose $N_0$ large enough so that for $n \geq N_0$, with probability at most $\eta$, $|\#\{F \,
\text{in } B\} - p^{e} n^v| \geq \epsilon n^v$.  Lastly, by Lemma~\ref{lem:constrBfailsexp}, we can
choose $N_0$ large enough so that for $n \geq N_0$, with probability at least $1-\eta$, there exists
a constant $C > 0$ and sets $S'_1,\dots,S'_t$ so that $|e(S'_1,\dots,S'_t) - p|S'_1|\dots|S'_t| | >
C \binom{n}{k}$.

Now fix $N_0$ to be the maximum of the $N_0$ from Lemma~\ref{lem:Bispdense}, the $N_0$ from
Lemma~\ref{lem:constrBfailsexp}, and the at most $2^{v_0}$ constants $N_0$ from
Lemma~\ref{lem:constrBsatisfiesexp}.  Note that the definition of $N_0$ depends only on $\epsilon$.
Now consider $n \geq N_0$ and let $B = B_{\vec{\pi}'}(n,p)$.  With probability at most $\eta$, $|
|E(B)| - p\binom{n}{k}| \geq \epsilon n^k$.  Also, with probability at most $2^{2^{v_0}} \eta =
\frac{1}{2}$, there exists some $F$ with at most $v_0$ vertices with $|\#\{F \, \text{in } B\} -
p^{e} n^v| \geq \epsilon n^v$.  By Lemma~\ref{lem:constrBfailsexp}, with probability at least
$1-\eta$, there exists a constant $C$ and sets $S'_1,\dots,S'_t$ such that $|e(S'_1,\dots,S'_t) -
p|S'_1|\dots|S'_t| | > C \binom{n}{k}$.  Therefore, for all $n \geq N_0$, with positive probability
a hypergraph drawn from the distribution $B_{\vec{\pi}'}(n,p)$ satisfies all three conditions in the
lemma.
\end{proof} 

\begin{lemma} \label{lem:expnotexp}
  If $0 < p < 1$ and if $\pi'$ is not a refinement of $\pi$ then \propexpandp{$\pi$}
  $\not \Rightarrow$ \propexpandp{$\pi'$}.
\end{lemma}

\begin{proof} 
Form a sequence of hypergraphs $\{B_{n_q}\}_{q\rightarrow\infty}$ as follows.  Let $\epsilon = 1/q$
and let $p_q$ be a rational with $|p_q - p| < \epsilon$.   Apply Lemma~\ref{lem:satisfypifailpi'} to
$p_q$ and $\epsilon$ to produce an $N_0(1/q)$.  Let $n_q$ be the maximum of $N_0$ and
$|V(B_{n_{q-1}})| + 1$.  Now since $n_q \geq N_0(1/q)$, Lemma~\ref{lem:satisfypifailpi'} guarantees
a hypergraph $B_{n_q}$ on $n_q$ vertices.  The sequence $\{B_{n_q}\}_{q\rightarrow\infty}$ will
satisfy \propcount{$\pi$-linear}; indeed, given any $\delta > 0$ and given any $\pi$-linear
hypergraph $F$ with $v$ vertices and $e$ edges, let $q_0 = \max\{\frac{2e}{\delta}, v\}$.  For all
$q \geq q_0$, we have that $|p_q - p| \leq \frac{1}{q} \leq \frac{\delta}{2e}$.  In addition, since
$v \leq q_0 \leq q$, we have that the number of labeled copies of $F$ in $B_{n_q}$ differs from
$p_q^e n_q^v$ by at most  $\frac{\delta}{2e} n_q^v$.  Using that $|p_q^e - p^e| \leq e|p_q -
p|$,\footnote{For $0 \leq a < b \leq 1$, $b^e - a^e = (b-a)(b^{e-1} + ab^{e-2} + a^2b^{e-3} + \dots
+ a^{e-2}b + a^{e-1}) \leq (b-a)(1+\dots+1)=(b-a) e$.} we have $|p_q^e - p^e| \leq
\frac{\delta}{2}$.  Therefore, the number of labeled copies of $F$ differs from $p^e n_q^v$ by at
most $(\frac{\delta}{2} + \frac{\delta}{2e}) n_q^v$.  Thus the number of labeled copies of $F$
differs from $p^e n_q^v$ by at most $\delta n_q^v$, which implies that \propcount{$\pi$-linear}
holds for the sequence $\{B_{n_q}\}_{q \rightarrow\infty}$.  By Lemma~\ref{lem:satisfypifailpi'},
there exists a constant $C > 0$ depending only on $p$, $k$, $\pi'$, and $t$ such that the hypergraph
$B$ from Lemma~\ref{lem:satisfypifailpi'} has $|e(S'_1,\dots,S'_t) - p|S'_1|\dots|S'_t| | > C
\binom{n}{k}$.  This implies that the sequence $\{B_{n_q}\}_{q\rightarrow\infty}$ fails
\propexpand{$\pi'$}.

Note that the sequence can be extended to have a hypergraph on $n$ vertices for every $n$.  If there
is a gap between $N_0(1/q)$ and $n_{q-1}+1$, Lemma~\ref{lem:satisfypifailpi'} can be applied many
times for $\epsilon = 1/(q-1)$ to fill in the gap.  Since $n_{q-1}$ was chosen bigger than
$N_0(1/(q-1))$, Lemma~\ref{lem:satisfypifailpi'} guarantees a hypergraph for every $n$ bigger than
$n_{q-1}$, in particular the integers between $n_{q-1}$ and $N_0(1/q)$.
\end{proof} 

\begin{lemma} \label{lem:expnotcd}
  For all $0 < p < 1$ and all $\pi$, \propexpand{$\pi$} $\not\Rightarrow$ \propcd{$2$}.
\end{lemma}

\begin{proof} 
Use a diagonalization argument similar to Lemmas~\ref{lem:satisfypifailpi'} and~\ref{lem:expnotexp}
based on $A_2(n,p)$.  By Lemma~\ref{lem:constrAsatisfiesexp}, with high probability $A_2(n,p)$
satisfies \propexpand{$\pi$} and by Lemma~\ref{lem:constrAfailsCD} fails \propcd{$2$}.
\end{proof} 

\begin{lemma} \label{lem:expnotdev}
  For $p = \frac{1}{2}$ and all $\pi$, \propexpand{$\pi$} $\not\Rightarrow$ \propdev{$2$}.
\end{lemma}

\begin{proof} 
Use a diagonalization argument similar to Lemmas~\ref{lem:satisfypifailpi'} and~\ref{lem:expnotexp}
based on $D(n,1/2)$.  By Lemma~\ref{lem:constrDsatsfiesexp}, with high probability $D(n,1/2)$
satisfies \propexpand{$\pi$} and by Lemma~\ref{lem:constrDfailsdev} fails \propdev{$2$}.
\end{proof} 

\begin{lemma} \label{lem:cdnotcd}
  For $0 < p < 1$ and $1 \leq \ell \leq k -2$, \propcd{$\ell$} $\not\Rightarrow$ \propcd{$\ell+1$}.
\end{lemma}

\begin{proof} 
This was proved by Chung~\cite{hqsi-chung90} for $p = \frac{1}{2}$ using a construction
similar to $A_{\ell+1}(n,1/2)$ except the random hypergraph is replaced by the Paley hypergraph.  We
expand the proof to all $0 < p < 1$ using a diagonalization argument similar to
Lemmas~\ref{lem:satisfypifailpi'} and~\ref{lem:expnotexp} based on $A_{\ell+1}(n,p)$.  By
Lemma~\ref{lem:constrAsatCD}, with high probability $A_{\ell+1}(n,p)$ satisfies \propcd{$\ell$} and
by Lemma~\ref{lem:constrAfailsCD} fails \propcd{$\ell+1$}.
\end{proof} 

\begin{lemma} \label{lem:cdnotdev}
  For $p = \frac{1}{2}$, \propcd{$k-2$} $\not\Rightarrow$ \propdev{$2$}.
\end{lemma}

\begin{proof} 
Let $\vec{\pi} = (k-1,1)$.  Use a diagonalization argument similar to
Lemmas~\ref{lem:satisfypifailpi'} and~\ref{lem:expnotexp} based on $B_{\vec{\pi}}(n,1/2)$.  By
Lemma~\ref{lem:constrBsatisfiesCD}, with high probability $B_{\vec{\pi}}(n,1/2)$ satisfies
\propcd{$k-2$} and by Lemma~\ref{lem:constrBfailsdev} fails \propdev{$2$}.
\end{proof} 

\begin{lemma} \label{lem:cdnotexp}
  For $0 < p < 1$, $2 \leq \ell \leq k -1$, and $\pi = k_1 + \dots + k_t$ with $k_i > \ell$ for some
  $i$, we have \propcd{$\ell$} $\not\Rightarrow$ \propexpand{$\pi$}.
\end{lemma}

\begin{proof} 
Let $\vec{\pi}' = (\ell+1,1,\dots,1)$.  Use a diagonalization argument similar to
Lemmas~\ref{lem:satisfypifailpi'} and~\ref{lem:expnotexp} based on $B_{\vec{\pi}'}(n,p)$.  By
Lemma~\ref{lem:constrBsatisfiesCD}, with high probability $B_{\vec{\pi}'}(n,p)$ satisfies
\propcd{$\ell$} and by Lemma~\ref{lem:constrBfailsexp} (for $B_{\vec{\pi}'}(n,p)$) fails
\propexpand{$\pi'$}.  Since $\pi'$ is a refinement of $\pi$, this implies that with high probability
$B_{\vec{\pi}'}(n,p)$ fails \propexpand{$\pi$}.
\end{proof} 

\begin{lemma} \label{lem:devnotcd} (Chung~\cite{hqsi-chung90})
  For $p = \frac{1}{2}$ and $2 \leq \ell \leq k -1$, \propdev{$\ell$} $\not\Rightarrow$
  \propcd{$\ell$}.
\end{lemma}

\begin{proof} 
This was originally proved by Chung~\cite{hqsi-chung90} using a construction similar to
$A_\ell(n,1/2)$ except the random hypergraph was replaced by the Paley hypergraph.
Lemma~\ref{lem:constrAsatdev} shows that with high probability $A_{\ell}(n,1/2)$ satisfies
\propdev{$\ell$} and Lemma~\ref{lem:constrAfailsCD} shows that $A_\ell(n,1/2)$ fails
\propcd{$\ell$}, so a diagonalization argument similar to Lemmas~\ref{lem:satisfypifailpi'}
and~\ref{lem:expnotexp} shows that $A_\ell(n,1/2)$ provides an alternate construction proving that
\propdev{$\ell$} $\not\Rightarrow$ \propcd{$\ell$}.
\end{proof} 

\begin{lemma} \label{lem:devnotdev}
  For $p = \frac{1}{2}$ and $2 \leq \ell \leq k-1$, \propdev{$\ell$} $\not\Rightarrow$
  \propdev{$\ell+1$}.
\end{lemma}

\begin{proof} 
Use a diagonalization argument similar to Lemmas~\ref{lem:satisfypifailpi'} and~\ref{lem:expnotexp}
based on $A_{\ell}(n,1/2)$.  By Lemma~\ref{lem:constrAsatdev}, with high probability
$A_{\ell}(n,1/2)$ satisfies \propdev{$\ell$} and by Lemma~\ref{lem:constrAfailsdev} fails
\propdev{$\ell+1$}.
\end{proof} 

\section{\propcd{$\ell,s$} } 
\label{sec:cdells}

In this section, we prove Theorem~\ref{thm:cdellsequiv}.  Initially, in \cite{hqsi-chung90}, Chung
claimed that \propcd{$\ell-1$} $\Leftrightarrow$ \propdev{$\ell$} for all $\ell$.  This fact is true
for $\ell = k$ since both are equivalent to \propcount{All}, but the claimed proof that
\propcd{$\ell-1$} $\Rightarrow$ \propdev{$\ell$} for $\ell < k$ was found to contain an error.  In
\cite{hqsi-chung12}, Chung discussed the error, proposed the property \propcd{$\ell,s$}, and claimed
that \propdev{$s$} $\Leftrightarrow$ \propcd{$k-1,s$}.  As our results (and
Lemma~\ref{lem:devnotimpcdells} below) will show, \propdev{$s$} $\not\Rightarrow$ \propcd{$k-1,s$}
and so there is an error in \cite{hqsi-chung12} (the error is in the second to last equality in the
equation at the end of Section~3 in \cite{hqsi-chung12}).  In fact, the following counterexample was
essentially discovered by Chung but our use of the random graph instead of the Paley graph makes
the construction simpler to analyze.  We note here that Chung's definition of \propcd{$\ell,s$}
considered spanning hypergraphs while we consider not necessarily spanning hypergraphs.  The
counterexample below works with either definition.

\begin{lemma} \label{lem:devnotimpcdells}
  For $k = 3$, \propdev{$2$} $\not\Rightarrow$ \propcdh{$2,2$}
\end{lemma}

\begin{proof}[Proof Sketch] 
Consider $A_2(n,1/2)$.  By Lemma~\ref{lem:constrAsatdev}, with high probability $A_2(n,1/2)$
satisfies \propdev{$2$}.  On the other hand, with high probability $A_2(n,1/2)$ will fail
\propcd{$2,2$} as follows.  Let $G$ be the (spanning) graph consisting of the edges colored one in the
definition of $A_2(n,1/2)$, so that a triple $T \in \binom{V(A_2(n,1/2))}{3}$ is a hyperedge if and only if
$|E(G[T])|$ is even.  The probability that $G[T]$ has no edges is $\frac{1}{8}$, has one edge is
$\frac{3}{8}$, has two edges is $\frac{3}{8}$, and has all three edges is $\frac{1}{8}$.  Thus
w.h.p.\ there are a total of $(1+o(1)) \frac{1}{2} \binom{n}{3}$ triples which induce at least two
edges of $G$.  But of these, only the ones with exactly two edges are hyperedges, so there are
$(1+o(1)) \frac{3}{8} \binom{n}{3}$ hyperedges of $A_2(n,1/2)$ inducing at least two edges of $G$.
But $\frac{3}{8}$ is not one-half of $\frac{1}{2}$, implying that \propcd{$2,2$} does not hold for
$A_2(n,1/2)$.
\end{proof} 


We now turn to proving Theorem~\ref{thm:cdellsequiv}, which states that the properties
\propcd{$\ell,s$} are equivalent for fixed $k$ and $\ell$ as $s$ ranges between $1$ and
$\binom{k}{\ell}$.  The proof occurs in two stages; first we prove that \propcd{$\ell,s+1$}
$\Rightarrow$ \propcd{$\ell,s$} and secondly prove that \propcd{$\ell,1$} $\Rightarrow$
\propcd{$\ell,\binom{k}{\ell}$}.  The former proof is the difficult one, and the main tool used in
the proof is inclusion/exclusion.

\begin{thm} \label{thm:inclusionexclusion}
  (Inclusion/Exclusion) Let $U$ be a finite set and let $f,g : 2^U \rightarrow \mathbb{N}$.
  If for all $A \subseteq U$,
  \begin{align*}
    g(A) = \sum_{B : B \subseteq A} f(B),
  \end{align*}
  then for all $A \subseteq U$,
  \begin{align*}
    f(A) = \sum_{B : B \subseteq A} (-1)^{|A| - |B|} g(B).
  \end{align*}
\end{thm}

\begin{definition}
  Let $G$ be an $\ell$-uniform hypergraph, let $\mathcal{P} = (P_1,\dots,P_k)$ be an ordered partition of $V(G)$
  into $k$ parts, and let $\mathcal{R} \subseteq \binom{[k]}{\ell}$.  Define an $\ell$-uniform
  hypergraph $G_{\mathcal{P},\mathcal{R}}$ as follows.  $V(G_{\mathcal{P},\mathcal{R}}) = V(G)$ and
  \begin{align*}
    E(G_{\mathcal{P},\mathcal{R}}) = \left\{ X \in E(G) : \{i : X \cap P_i \neq \emptyset\} \in
    \mathcal{R} \right\}.
  \end{align*}
\end{definition}

Conceptually, $G_{\mathcal{P}, \mathcal{R}}$ is the subhypergraph of $G$ consisting of those
edges with at most one vertex in each part of the $k$-partition $\mathcal{P}$ where
in addition the intersection pattern of the edge appears in $\mathcal{R}$.  Our proof that
\propcd{$\ell,s+1$} $\Rightarrow$ \propcd{$\ell,s$} works as follows: given some $\ell$-uniform
hypergraph $G$, we modify $G$ so that the $k$-sets inducing at least $s$ edges of $G$ transition to
$k$-sets inducing at least $s+1$ edges in the modification of $G$.  The complexity in the proof is
that the modification of $G$ must be carefully chosen so that there is a strong relationship between
the $k$-sets inducing $s+1$ edges in the modification and the $k$-sets inducing at least $s$ edges
of $G$.  This modification uses $G_{\mathcal{P}, \mathcal{R}}$ as follows: pick some $I \in
\binom{[k]}{\ell}$ with $I \notin \mathcal{R}$ and define $F$ to be the $\ell$-uniform hypergraph
$G_{\mathcal{P}, \mathcal{R}}$ plus the complete $\ell$-partite, $\ell$-uniform hypergraph with
edges whose intersection pattern on $\mathcal{P}$ is given by $I$.

Now consider applying \propcd{$\ell,s+1$} to $F$, which tells us about the $k$-sets inducing at
least $s+1$ edges of $F$.  The $k$-sets which contain exactly one vertex in each part of
$\mathcal{P}$ are well behaved.  Indeed, if $|\mathcal{R}| = s$ and $T$ is a $k$-set with exactly
one vertex in each part of $\mathcal{P}$, then $T$ will induce at least $s$ edges of
$G_{\mathcal{P},\mathcal{R}}$ if and only if $T$ induces exactly $s$ edges of
$G_{\mathcal{P},\mathcal{R}}$ since there are only $s$ intersection patterns in $\mathcal{R}$.  In
this case, $T$ will induce exactly $s+1$ edges of $F$ since $F$ added the complete $\ell$-partite
hypergraph in intersection pattern $I$ and $I \notin \mathcal{R}$.  Applying \propcd{$\ell,s+1$} to
$F$ also tells us about $k$-sets which have more than one vertex in some part of $\mathcal{P}$, but
an inclusion/exclusion argument is used to ignore these $k$-sets.  In summary, we restrict from $G$
to $G_{\mathcal{P}, \mathcal{R}}$ so that we have room to add a complete $\ell$-partite graph of
intersection pattern $I$ without interfering with the edges of $G$, and use inclusion/exclusion argument to
study only the $k$-sets with exactly one vertex in each part since only for these $k$-sets can we
transfer knowledge between $G_{\mathcal{P}, \mathcal{R}}$ and $F$.  The next definition gives a
symbol to these $k$-sets with exactly one vertex in each part which induce exactly $s$ edges of $G$.

\begin{definition}
  Let $G$ be an $\ell$-uniform hypergraph, let $\mathcal{P} = (P_1,\dots,P_k)$ be an ordered partition of $V(G)$
  into $k$ parts, and let $s$ and $k$ be integers where $k > \ell$ and $1 \leq s \leq \binom{k}{\ell}$.  Define
  \begin{align*}
    W(G,\mathcal{P},s) = \left\{ T \in \binom{V(G)}{k} : \forall i, T \cap P_i \neq \emptyset
    \,\,\text{and } e_G(T) = s \right\},
  \end{align*}
  where $e_G(T) = |E(G[T])|$ is the number of edges of $G$ induced by $T$.
\end{definition}

\begin{lemma} \label{lem:cdellspartite}
  Let $k$, $\ell$, and $s$ be integers with $2 \leq \ell < k$ and $1 \leq s < \binom{k}{\ell}$.
  Let $\mathcal{H} = \{H_n\}_{n\rightarrow\infty}$ be a sequence of $k$-uniform hypergraphs with
  $|V(H_n)| = n$ and assume $\mathcal{H}$ satisfies \propcdp{$\ell,s+1$}.  Let $G$ be an $\ell$-uniform
  hypergraph with $V(G) \subseteq V(H_n)$, let $\mathcal{P} = (P_1,\dots,P_k)$ be an ordered partition of $V(G)$ into $k$ parts, and
  let $\mathcal{R} \subseteq \binom{[k]}{\ell}$ where $\left| \mathcal{R} \right| = s$.  Then
  \begin{align*}
    \left| W(G_{\mathcal{P},\mathcal{R}},\mathcal{P},s) \cap E(H_n) \right|
    = p \left| W(G_{\mathcal{P},\mathcal{R}},\mathcal{P},s) \right| + o(n^k).
  \end{align*}
\end{lemma}

\begin{proof} 
Throughout this proof, the subscripts $n$ and $p$ are dropped for clarity.  Since $s <
\binom{k}{\ell}$, pick some $I \in \binom{[k]}{\ell}$ where $I \notin \mathcal{R}$.  Define
\begin{align*}
  F = G_{\mathcal{P},\mathcal{R}} \cup \left\{ X \in \binom{V(G)}{\ell} : \{i : X \cap P_i
  \neq \emptyset \} = I \right\}.
\end{align*}
Now define maps $f_{K_n}, f_H,  g_{K_n}, g_H : 2^{[k]} \rightarrow \mathbb{N}$ as follows:
\begin{align*}
  f_{K_n}(A) &= \left| \left\{ 
    T \in \binom{V(G)}{k} : e_F(T) \geq s+1, \{ i : T \cap P_i \neq \emptyset\} = A 
    \right\} \right|, \\
  f_H(A) &= \left| \left\{ 
    T \in E(H) \quad \,\,\,\, : e_F(T) \geq s+1, \{ i : T \cap P_i \neq \emptyset\} = A 
    \right\} \right|, \\
  g_{K_n}(A) &= \left| \left\{ 
    T \in \binom{V(G)}{k} : e_F(T) \geq s+1, \{ i : T \cap P_i \neq \emptyset\} \subseteq A
    \right\} \right|, \\
  g_{H}(A) &= \left| \left\{ 
    T \in E(H) \quad \,\,\,\, : e_F(T) \geq s+1, \{ i : T \cap P_i \neq \emptyset\} \subseteq A
    \right\} \right|.
\end{align*}
Note that in the above definitions, the set $T$ could have more than one vertex in each part of
$\mathcal{P}$.

\newcounter{cdellsctr}
\newtheorem{cdellsclaim}[cdellsctr]{Claim}

\begin{cdellsclaim} \label{c:gispdense}
  For all $B \subseteq [k]$, $g_H(B) = p \, g_{K_n}(B) + o(n^k)$.
\end{cdellsclaim}

\begin{proof} 
This will follow from applying \propcd{$\ell,s+1$} to $F' = F[\cup_{i\in B} P_i]$ as follows.
Consider the sets
\begin{align*}
  \Delta_1
  &= \left\{ T \in \binom{V(F')}{k} : e_{F'}(T) \geq s+1 \right\}, \\
  \Delta_2
  &= \left\{ T \in \binom{V(G)}{k} : e_F(T) \geq s+1, \{ i : T \cap P_i \neq \emptyset\} \subseteq B\right\}.
\end{align*}
Since $F'$ is $F$ restricted to $\cup_{i\in B} P_i$, $\Delta_1 = \Delta_2$.  Next, applying
\propcd{$\ell,s+1$} to $F'$ shows that $|\Delta_1 \cap E(H)| = p|\Delta_1| + o(n^k)$.  Lastly, by
the definitions of $g_H$ and $g_{K_n}$, $g_H(B) = |\Delta_2 \cap E(H)|$ and $g_{K_n}(B) =
|\Delta_2|$.  Since $\Delta_1 = \Delta_2$, the proof is complete.
\end{proof} 

\begin{cdellsclaim} \label{c:gissumf}
  \begin{align*}
    g_{K_n}(A) = \sum_{B : B\subseteq A} f_{K_n}(B) \quad \quad \quad
    g_{H}(A) = \sum_{B : B\subseteq A} f_{H}(B) \quad \quad \quad
  \end{align*}
\end{cdellsclaim}

\begin{proof} 
Let $T \in \binom{V(G)}{k}$ with $e_F(T) \geq s+1$ and $\{i : T \cap
P_i \neq \emptyset \} \subseteq A$.  Define $B = \{ i : T \cap P_i \neq
\emptyset\}$ so that $B \subseteq A$.  Now $T$ will be counted once by $g_{K_n}(A)$ and once by
$f_{K_n}(B)$ but will not be counted by any $f_{K_n}(B')$ with $B' \neq B$.  A similar argument
shows that if $T \in E(H)$, $T$ will be counted once by $g_H(A)$ and once by $f_H(B)$.
\end{proof} 

\begin{cdellsclaim} \label{c:fisgsum}
  \begin{align*}
    f_{K_n}(A) = \sum_{B : B\subseteq A} (-1)^{|A| - |B|} g_{K_n}(B) \quad \quad \quad
    f_{H}(A) = \sum_{B : B\subseteq A} (-1)^{|A| - |B|} g_{H}(B) \quad \quad \quad
  \end{align*}
\end{cdellsclaim}

\begin{proof} 
Apply Claim~\ref{c:gissumf} and Inclusion/Exclusion (Theorem~\ref{thm:inclusionexclusion}) to $f_{K_n},
g_{K_n}$ and $f_H, g_H$.
\end{proof} 

\begin{cdellsclaim} \label{c:fispdense}
  For all $A \subseteq [k]$, $f_H(A) = p \, f_{K_n}(A) + o(n^k)$.
\end{cdellsclaim}

\begin{proof} 
Combine Claims~\ref{c:gispdense} and~\ref{c:fisgsum} to obtain
\begin{align*}
  f_H(A) = \sum_{B : B \subseteq A} (-1)^{|B| - |A|} g_H(B) 
  = p \sum_{B : B \subseteq A} (-1)^{|B| - |A|} g_{K_n}(B) + o(n^k)
  = p \, f_{K_n}(A) + o(n^k).
\end{align*}
\end{proof} 

Claim~\ref{c:fispdense} for $A = [k]$ implies that among the $k$-sets $T$ with exactly one vertex in
each part and inducing at least $s+1$ edges of $F$, a $p$-fraction of them are hyperedges of $H$.  The
remainder of the proof translates this knowledge back to $k$-sets inducing at least $s$ edges of
$G_{\mathcal{P}, \mathcal{R}}$, using that $F$ was built from $G_{\mathcal{P}, \mathcal{R}}$ by
adding the complete $\ell$-partite, $\ell$-uniform hypergraph with intersection pattern $I$.

\begin{cdellsclaim} \label{c:fkisWf}
  \begin{align*}
    f_{K_n}([k]) = \left| W(F,\mathcal{P},s+1) \right|
    \quad \quad \text{and} \quad \quad
    f_{H}([k]) = \left| W(F,\mathcal{P},s+1) \cap E(H) \right|
  \end{align*}
\end{cdellsclaim}

\begin{proof} 
First, we show that $f_{K_n}([k]) \leq \left| W(F,\mathcal{P},s+1) \right|$.  Let $T \in
\binom{V(G)}{k}$ with $e_F(T) \geq s+1$ and $T \cap P_i \neq \emptyset$ for all $i \in [k]$, so that
$T$ is counted by $f_{K_n}([k])$.  Now consider the set $\mathcal{R}' = \{ J : \exists X \in
E(F[T]), \{ i : X \cap P_i \neq \emptyset\} = J\}$.  Since $T$ has exactly one vertex in each
$P_i$, $|\mathcal{R}'| \geq s+1$ and also every $J \in \mathcal{R}'$ has size $\ell$.  By the
definition of $F$, $\mathcal{R}' \subseteq \mathcal{R} \cup \{ I \}$ which when combined with
$|\mathcal{R}| = s$ shows that $\mathcal{R}' = \mathcal{R} \cup \{ I \}$.  In particular,
$|\mathcal{R}'| = s+1$ so $e_F(T) = s+1$, which implies that $T \in W(F,\mathcal{P},s+1)$ and thus
$f_{K_n}([k]) \leq \left| W(F,\mathcal{P},s+1) \right|$.

Next, we prove that $f_{K_n}([k]) \geq \left| W(F,\mathcal{P},s+1) \right|$.  Let $T \in
W(F,\mathcal{P},s+1)$.  Then $e_F(T) = s+1$ and $|T \cap P_i| = 1$ for all $i \in [k]$ so $\{i : T
\cap P_i \neq \emptyset\} = [k]$.  Thus $T$ is counted by $f_{K_n}([k])$ so $f_{K_n}([k]) = \left|
W(F,\mathcal{P},s+1) \right|$.

A similar argument shows that $f_{H}([k]) = \left| W(F,\mathcal{P},s+1) \cap E(H) \right|$, since
the previous two paragraphs can be applied to sets $T$ which are edges of $H$.
\end{proof} 

\begin{cdellsclaim} \label{c:WfisWg}
  \begin{align*}
    \left| W(F,\mathcal{P},s+1) \right|
    =  \left| W(G_{\mathcal{P},\mathcal{R}},\mathcal{P},s) \right|
    \text{and} 
    \left| W(F,\mathcal{P},s+1) \cap E(H) \right|
    =  \left| W(G_{\mathcal{P},\mathcal{R}},\mathcal{P},s) \cap E(H) \right|.
  \end{align*}
\end{cdellsclaim}

\begin{proof} 
Let $T \in \binom{V(G)}{k}$ with $|T \cap P_i| = 1$ for all $i$.  We would like to show that $e_F(T)
= s+1$ if and only if $e_{G_{\mathcal{P},\mathcal{R}} }(T) = s$.  As in the previous proof, define
$\mathcal{R}' = \{ J : \exists X \in E(F[T]), \{ i : X \cap P_i \neq \emptyset\} = J\}$.  Now
$e_F(T) = s+1$ if and only if $\mathcal{R}' = \mathcal{R} \cup \{ I \}$.  Since $F$ is defined as
the edges of $G_{\mathcal{P},\mathcal{R}}$ together with all $\ell$-sets with intersection pattern
$I$, $\mathcal{R}' = \mathcal{R} \cup \{ I \}$ if and only if $e_{G_{\mathcal{P},\mathcal{R}} }(T) =
s$.  This implies that $\left| W(F,\mathcal{P},s+1) \right| =  \left|
W(G_{\mathcal{P},\mathcal{R}},\mathcal{P},s) \right|$.  A similar argument where $T$ is restricted
to an edge of $H$ shows that $\left| W(F,\mathcal{P},s+1) \cap E(H) \right| =  \left|
W(G_{\mathcal{P},\mathcal{R}},\mathcal{P},s) \cap E(H) \right|$.
\end{proof} 

We can now complete the proof of Lemma~\ref{lem:cdellspartite}.  Combining Claims~\ref{c:fkisWf}
and~\ref{c:WfisWg} shows that
\begin{align*}
  f_{K_n}([k]) 
  &=  \left| W(G_{\mathcal{P},\mathcal{R}},\mathcal{P},s) \right| \\
  f_H([k])
  &=  \left| W(G_{\mathcal{P},\mathcal{R}},\mathcal{P},s) \cap E(H) \right|.
\end{align*}
Claim~\ref{c:fispdense} then shows that
\begin{align*}
  \left| W(G_{\mathcal{P},\mathcal{R}},\mathcal{P},s) \cap E(H) \right|
  = p \left| W(G_{\mathcal{P},\mathcal{R}},\mathcal{P},s) \right| + o(n^k),
\end{align*}
completing the proof of the lemma.
\end{proof} 

\begin{lemma} \label{lem:cdellsp1impcdells}
  Let $k$, $\ell$, and $s$ be integers with $2 \leq \ell < k$ and $1 \leq s < \binom{k}{\ell}$.
  Then \propcdp{$\ell, \linebreak[1] s+1$} $\Rightarrow$ \propcdp{$\ell,s$}.
\end{lemma}

\begin{proof} 
Let $\mathcal{H} = \{H_n\}_{n\rightarrow\infty}$ be a sequence of $k$-uniform hypergraphs with
$|V(H_n)| = n$ and assume $\mathcal{H}$ satisfies \propcd{$\ell,s+1$}.  Let $G$ be an
$\ell$-uniform hypergraph with $V(G) \subseteq V(H)$ and let $n' = |V(G)|$.  Then
\begin{align} \label{eq:cdellsovercount}
  \left| \left\{ T \in \binom{V(G)}{k} : e_G(T) = s \right\} \right| = \frac{1}{k!k^{n'-k}}
  \sum_{\mathcal{P},\mathcal{R}} \left| W(G_{\mathcal{P},\mathcal{R}},\mathcal{P},s) \right|.
\end{align}
Indeed, let $T = \{ t_1,\dots,t_k\} \subseteq V(H)$ with
$e_G(T) = s$. The number of times $T$ is counted in the sum is $k! k^{n'-k}$ since $T$ will be
counted on the right hand side of \eqref{eq:cdellsovercount} only if $T \cap P_i \neq \emptyset$ for
each $i$.  There are $k! $ ways of assigning the vertices of $T$ to the parts of $\mathcal{P}$, and
$k^{n'-k}$ ways of assigning the other $n'-k$ vertices of $G$ to parts of $\mathcal{P}$.  Once such
a partition $\mathcal{P}$ is chosen, there is a unique choice for $\mathcal{R}$ since $T$ induces
exactly $s$ edges and $T$ has exactly one vertex in each $P_i$. A similar counting argument
shows that
\begin{align*}
  \left| \left\{ T \in E(H) : e_G(T) = s \right\} \right| = \frac{1}{k!k^{n'-k}}
  \sum_{\mathcal{P},\mathcal{R}} \left| W(G_{\mathcal{P},\mathcal{R}},\mathcal{P},s) \cap E(H)
  \right|.
\end{align*}
Note that the number of terms in the sum is $k! S(n',k) \binom{\binom{k}{\ell}}{s} =
\Theta(k^{n'})$, so applying Lemma~\ref{lem:cdellspartite} implies that
\begin{align}
  \left| \left\{ T \in E(H) : e_G(T) = s \right\} \right| 
  &= \frac{1}{k!k^{n'-k}} \sum_{\mathcal{P},\mathcal{R}}
     \left| W(G_{\mathcal{P},\mathcal{R}},\mathcal{P},s) \cap E(H) \right| \nonumber \\
  &= \frac{p}{k!k^{n'-k}} \sum_{\mathcal{P},\mathcal{R}}
     \left| W(G_{\mathcal{P},\mathcal{R}},\mathcal{P},s) \right|
     + o \left( \frac{n^k k! S(n',k) \binom{\binom{k}{\ell}}{s}}{k!k^{n'-k}} \right) \nonumber \\
  &= p\left| \left\{ T \in \binom{V(G)}{k} :
  e_G(T) = s \right\} \right| + o(n^k). \label{eq:cdellsexactlys}
\end{align}
Applying \propcd{$\ell,s+1$} to $G$ shows that
\begin{align} \label{eq:cdellsatleastsp1}
  \left| \left\{ T \in E(H) : e_G(T) \geq s+1 \right\} \right| = p\left| \left\{ T \in \binom{V(G)}{k} :
  e_G(T) \geq s+1 \right\} \right| + o(n^k).
\end{align}
Combining \eqref{eq:cdellsexactlys} and \eqref{eq:cdellsatleastsp1} shows that
\begin{align*}
  \left| \left\{ T \in E(H) : e_G(T) \geq s \right\} \right| = p\left| \left\{ T \in \binom{V(G)}{k} :
  e_G(T) \geq s \right\} \right| + o(n^k),
\end{align*}
implying that \propcd{$\ell,s$} holds.
\end{proof} 

\begin{lemma} \label{lem:cdequivcomplement}
  \propcdp{$\ell,1$} $\Rightarrow$ \propcdp{$\ell,\binom{k}{\ell}$} 
\end{lemma}

\begin{proof} 
Let $G$ be an $\ell$-uniform hypergraph with $V(G) \subseteq V(H)$ and denote by $\bar G$ the
hypergraph with $V(\bar G) = V(G)$ and $E(\bar G) = \binom{V(G)}{\ell} -  E(G)$.  Now define
\begin{align*}
  A := \left\{ T \in \binom{V(G)}{k} : e_{\bar{G}}(T) \geq 1 \right\}.
\end{align*}
By the definition of $\bar{G}$,
\begin{align*}
  A = \left\{ T \in \binom{V(G)}{k} : e_{G}(T) < \binom{k}{\ell} \right\}.
\end{align*}
Thus if we define
\begin{align*}
  B := \left\{ T \in \binom{V(G)}{k} : e_G(T) = \binom{k}{\ell} \right\},
\end{align*}
then $\binom{V(G)}{k} - A = B$.  Intersecting this equation with $E(H)$, we also obtain that $E(H) -
(A \cap E(H)) = B \cap E(H)$. Apply \propcd{$\ell,1$} to $\bar G$ to imply that $|A \cap E(H)| = p |A|
+ o(n^k)$ which implies that
\begin{align*}
  |B \cap E(H)| = \Big(|E(H)| - |A \cap E(H)|\Big) = |E(H)| - p|A| + o(n^k).
\end{align*}
Since \propcd{$\ell,1$} applied to $K_n^{(\ell)}$ shows that $|E(H)| = p\binom{n}{k} + o(n^k)$,
we have
\begin{align*}
  |B \cap E(H)| = |E(H)| - p|A| + o(n^k) = p\binom{n}{k} - p|A| + o(n^k) = p|B| + o(n^k),
\end{align*}
which implies that \propcd{$\ell,\binom{k}{\ell}$} holds for $\mathcal{H}$.
\end{proof} 

Combining Lemmas~\ref{lem:cdellsp1impcdells} and~\ref{lem:cdequivcomplement} completes the proof of
Theorem~\ref{thm:cdellsequiv}.  Note that the proof of Lemma~\ref{lem:cdequivcomplement} extends to
show that \propcd{$\ell,s$} $\Leftrightarrow$ \propcd{$\ell,\binom{k}{\ell}-s+1$} by defining $A$ as
the $k$-sets inducing at least $s$ edges of $\bar G$.  Also note that the proof of
Lemma~\ref{lem:cdequivcomplement}  works even when the definition of \propcd{$\ell,s$} is restricted
to spanning graphs as in Chung's~\cite{hqsi-chung12} original definition, so
Lemma~\ref{lem:cdequivcomplement} provides an alternate contradiction to \cite{hqsi-chung12}.

\medskip \emph{Acknowledgements.} The authors would like to thank the referee for helpful feedback.

\bibliographystyle{abbrv}
\bibliography{refs.bib}

\appendix

\section{Deviation} 
\label{sec:chungdeviation}

This section contains the proofs of Lemmas~\ref{lem:subdev} and~\ref{lem:devcauchy}.  The ideas
behind these two lemmas are the cornerstone of Chung's~\cite{hqsi-chung90} proofs that
\propdev{$\ell$} $\Rightarrow$ \propdev{$\ell-1$} and \propdev{$\ell$} $\Rightarrow$
\propcd{$\ell-1$}.  Since our proof that \propdev{$2$} $\Rightarrow$ \propexpand{$\pi$} (which
appears in Section~\ref{sub:implicationsdev}) is based on the these same ideas, we factored out
these two lemmas from Chung's~\cite{hqsi-chung90} proofs.  Chung doesn't explicitly state these
lemmas, so for completeness we give proofs in this section.

\begin{proof}[Proof of Lemma~\ref{lem:subdev}] 
Let $P,Q \subseteq V(H)^k$ and assume that $Q$ is complete in coordinate $i$ so that there exists a
$Q' \subseteq V(H)^{k-1}$ with $Q = \{ (x_1,\dots,x_k) : (x_1,\dots,x_{i-1},x_{i+1},\dots,x_k) \in
Q', x_i \in V(H) \}$.  By definition,
\begin{align*}
  \dev_{\ell,P}(H) &= \sum_{\substack{x_1,\dots,x_{k-\ell},y_{1,0},y_{1,1},\dots,y_{\ell,0},y_{\ell,1} \in V(H) \\
      \oct{x_1;\dots;x_{k-\ell};y_{1,0},y_{1,1};\dots;y_{\ell,0},y_{\ell,1}} \subseteq P}}
      \eta_H(x_1;\dots;x_{k-\ell};y_{1,0},y_{1,1};\dots;y_{\ell,0},y_{\ell,1}) \\
      &= \sum_{\substack{\vec{x} \in V(H)^{k-\ell} \\ \vec{y} \in V(H)^{2\ell} \\
      \mathcal{O}[\vec{x};\vec{y}] \subseteq P} } \eta(\vec{x};\vec{y}).
\end{align*}
where for notational convenience we write $\mathcal{O}[\vec{x};\vec{y}]$ for
$\mathcal{O}[x_1;\dots;x_{k-\ell};y_{1,0},y_{1,1};\dots;y_{\ell,0},y_{\ell,1}]$ and similarly for
$\eta$.  Let $j = i - k + \ell$ and rearrange the sum to obtain
\begin{align} \label{eq:subdevexpanded}
  \dev_{\ell,P}(H) =
    \sum_{\vec{x} \in V(H)^{k-\ell}}
    \,\,
    \sum_{\substack{y_{1,0},y_{1,1},\dots,y_{j-1,0},y_{j-1,1} \in V(H) \\
                    y_{j+1,0},y_{j+1,1},\dots,y_{\ell,0},y_{\ell,1} \in V(H)}}
    \,\,
    \sum_{\substack{y_{j,0},y_{j,1} \in V(H) \\ \mathcal{O}[\vec{x};\vec{y}] \subseteq P}}
    \eta(\vec{x};\vec{y}).
\end{align}
Now fix $\vec{x}$ and $y_{1, 0}, y_{1, 1}, \dots, y_{j-1, 0}, y_{j-1, 1}, y_{j+1, 0}, y_{j+1,
1}, \dots, y_{\ell, 0}, y_{\ell, 1}$ and consider the sum over $y_{j,0}, y_{j,1}$ in the above
expression.  Call a vertex $z$ \emph{even} if
\begin{align*}
  |\tilde{\mathcal{O}}[\vec{x};y_{1,0},y_{1,1};\dots;y_{j-1,0},y_{j-1,1};z;y_{j+1,0},y_{j+1,1};\dots;y_{\ell,0},y_{\ell,1}]
   \cap E(H)|
\end{align*}
is even and \emph{odd} otherwise.  In other words, $z$ is even if the squashed octahedron formed
using $z$ in the $i$th part is even.  Define $N = \{ z : \mathcal{O}[\vec{x};y_{1, 0}, y_{1,
1};\dots;y_{j-1, 0}, y_{j-1, 1};z;y_{j+1, 0}, y_{j+1, 1};\dots; \linebreak[1] y_{\ell, 0}, y_{\ell,
1}] \subseteq P\}$. Now expand the sum over $y_{j,0}$ and $y_{j,1}$ by cases depending on if the
vertices are even or odd.
\begin{align}
  \sum_{\substack{y_{j,0},y_{j,1} \in V(H) \\ \mathcal{O}[\vec{x};\vec{y}] \subseteq P}}
    \eta(\vec{x};\vec{y})
    = &\sum_{y_{j,0}, y_{j,1} \in N} \eta(\vec{x};\vec{y}) \nonumber \\
    = &\sum_{\substack{y_{j,0} \, \text{even in } N \\ y_{j,1} \, \text{even in } N}}
        \eta(\vec{x};\vec{y})
    +  \sum_{\substack{y_{j,0} \, \text{even in } N \\ y_{j,1} \, \text{odd in } N}}
        \eta(\vec{x};\vec{y}) \nonumber \\
    + &\sum_{\substack{y_{j,0} \, \text{odd in } N \\ y_{j,1} \, \text{even in } N}}
        \eta(\vec{x};\vec{y})
    +  \sum_{\substack{y_{j,0} \, \text{odd in } N \\ y_{j,1} \, \text{odd in } N}}
        \eta(\vec{x};\vec{y}). \label{eq:subdevoddeven}
\end{align}
If $y_{j,0}$ and $y_{j,1}$ are both even or both odd then $\eta(\vec{x};\vec{y}) = +1$ and if
exactly one of $y_{j,0}, y_{j,1}$ is even then $\eta(\vec{x};\vec{y}) = -1$.  Let $\Gamma_0$ be the
number of even vertices in $N$ and $\Gamma_1$ the number of odd vertices in $N$.  Then continuing
the above equation we have
\begin{align} \label{eq:subdevalwayspositive}
  \sum_{\substack{y_{j,0},y_{j,1} \in V(H) \\ \mathcal{O}[\vec{x};\vec{y}] \subseteq P}}
    \eta(\vec{x};\vec{y}) = \Gamma_0^2 - \Gamma_0 \Gamma_1 - \Gamma_1 \Gamma_0 + \Gamma_1^2
    = (\Gamma_0 - \Gamma_1)^2 \geq 0.
\end{align}
In particular, this implies that the above sum is always non-negative.  Now return to
\eqref{eq:subdevexpanded}.  Since the innermost sum is always non-negative for any choice of
$\vec{x}$ and $y_{1, 0}, y_{1, 1}, \dots, y_{j-1, 0}, \linebreak[1] y_{j-1, 1}, y_{j+1, 0}, y_{j+1,
1}, \dots, y_{\ell, 0}, y_{\ell, 1}$, the middle sum in \eqref{eq:subdevexpanded} can be restricted
to $Q'$ and this restriction cannot make the value of the sum go up.  More precisely,
\begin{align*}
  \dev_{\ell,P}(H) &=
    \sum_{\vec{x} \in V(H)^{k-\ell}}
    \,\,
    \sum_{\substack{y_{1,0},y_{1,1},\dots,y_{j-1,0},y_{j-1,1} \in V(H) \\
                    y_{j+1,0},y_{j+1,1},\dots,y_{\ell,0},y_{\ell,1} \in V(H)}}
    \,\,
    \sum_{\substack{y_{j,0},y_{j,1} \in V(H) \\ \mathcal{O}[\vec{x};\vec{y}] \subseteq P}}
    \eta(\vec{x};\vec{y}) \\
    &\geq
    \sum_{\vec{x} \in V(H)^{k-\ell}}
    \,\,
    \sum_{\substack{y_{1,0},y_{1,1},\dots,y_{j-1,0},y_{j-1,1} \in V(H) \\
                    y_{j+1,0},y_{j+1,1},\dots,y_{\ell,0},y_{\ell,1} \in V(H) \\
                    \mathcal{O}[\vec{x};y_{1,0},y_{1,1};\dots;y_{j-1,0},y_{j-1,1};y_{j+1,0}, y_{j+1,1}; 
                    \dots;y_{\ell,0}, y_{\ell,1}] \subseteq Q'
                    }}
    \,\,
    \sum_{\substack{y_{j,0},y_{j,1} \in V(H) \\ \mathcal{O}[\vec{x};\vec{y}] \subseteq P}}
    \eta(\vec{x};\vec{y}).
\end{align*}
Notice that the octahedron in the middle sum skips the $j$th coordinate of the $y$s which
corresponds to the $i$th coordinate of the octahedron.  This matches with the fact that $Q'
\subseteq V(H)^{k-1}$.  By definition of $Q'$, $\mathcal{O}[\vec{x}; y_{1,0},y_{1,1}; \dots;
y_{j-1,0},y_{j-1,1}; y_{j+1,0}, y_{j+1,1}; \dots; \linebreak[1] y_{\ell,0}, y_{\ell,1}] \subseteq
Q'$ if and only if $\mathcal{O}[\vec{x};\vec{y}] \subseteq Q$.  Thus the above sum simplifies to
\begin{align*}
  \dev_{\ell,P}(H) \geq \sum_{\vec{x} \in V(H)^{k-\ell}} \,\,
  \sum_{\substack{\vec{y} \in V(H)^{2\ell} \\ \mathcal{O}[\vec{x};\vec{y}] \subseteq P \cap Q}}
  \eta(\vec{x};\vec{y})
  = \dev_{\ell,P\cap Q}(H).
\end{align*}
\end{proof} 

Next we prove Lemma~\ref{lem:devcauchy}, which is a consequence of the Cauchy-Schwartz inequality.

\begin{thm} \label{thm:cauchyschwartz}
  (Cauchy-Schwartz Inequality)  If $n$ is a positive integer and $\alpha_i, \beta_i \in \mathbb{R}$ for
  $1 \leq i \leq n$, then
  \begin{align*}
    \left( \sum_{i=1}^n \alpha_i \beta_i \right)^2 
       \leq \left( \sum_{i=1}^n \alpha_i^2 \right)
            \left( \sum_{i=1}^n \beta_i^2 \right).
  \end{align*}
\end{thm}

\begin{proof}[Proof of Lemma~\ref{lem:devcauchy}] 
Let $\mathcal{H}$ be a sequence of hypergraphs where $\dev_{\ell,P}(H_n) = o(n^{k+\ell})$.
Using the same notation as the proof of Lemma~\ref{lem:subdev}, we have
\begin{align*}
  \dev_{\ell,P}(H_n) 
  &= \sum_{\vec{x} \in V(H_n)^{k-\ell}} 
  \,\,
  \sum_{\substack{\vec{y} \in V(H_n)^{2\ell} \\ \mathcal{O}[\vec{x};\vec{y}] \subseteq P}}
  \eta(\vec{x};\vec{y}) \\
  &= \sum_{\vec{x} \in V(H_n)^{k-\ell}} 
  \,\,
  \sum_{y_{2,0},y_{2,1},\dots,y_{\ell,0},y_{\ell,1} \in V(H_n)}
  \,\,
  \sum_{\substack{y_{1,0},y_{1,1} \in V(H_n) \\ \mathcal{O}[\vec{x};\vec{y}] \subseteq P}}
  \eta(\vec{x};\vec{y}). \\
\end{align*}
For $j = 1$, the equations \eqref{eq:subdevoddeven} and \eqref{eq:subdevalwayspositive} show
that
\begin{align*}
  \sum_{\substack{y_{1,0},y_{1,1} \in V(H_n) \\ \mathcal{O}[\vec{x};\vec{y}] \subseteq P}}
  \eta(\vec{x};\vec{y}) = (\Gamma_0 - \Gamma_1)^2,
\end{align*}
where $\Gamma_0$ is the number of even vertices in $N = \{ z :
\mathcal{O}[\vec{x};z;y_{2,0},y_{2,1};\dots;y_{\ell,0},y_{\ell,1}] \subseteq P \}$ and
$\Gamma_1$ is the number of odd vertices in $N$.  But by the definition of an even and odd vertex,
\begin{align*}
  \Gamma_0 - \Gamma_1 = 
  \sum_{z \in N} \eta(\vec{x};z;y_{2,0},y_{2,1};\dots;y_{\ell,0},y_{\ell,1}).
\end{align*}
Thus
\begin{align*}
  \dev_{\ell,P}(H_n) =
  \sum_{\vec{x} \in V(H_n)^{k-\ell}} 
  \,\,
  \sum_{y_{2,0},y_{2,1},\dots,y_{\ell,0},y_{\ell,1} \in V(H_n)}
  \,\,
  \left(  
  \sum_{z \in N} \eta(\vec{x};z;y_{2,0},y_{2,1};\dots;y_{\ell,0},y_{\ell,1})
  \right)^2.
\end{align*}
Now apply Cauchy-Schwartz with $\alpha_i = 1$ and $\beta_i = \sum_z \eta(\cdots)$ to obtain
\begin{align*}
  \dev_{\ell,P}(H_n) \geq
  \frac{1}{n^{k+\ell-2}}
  \left(  
  \sum_{\vec{x} \in V(H_n)^{k-\ell}} 
  \,\,
  \sum_{y_{2,0},y_{2,1},\dots,y_{\ell,0},y_{\ell,1} \in V(H_n)}
  \,\,
  \sum_{z \in N} \eta(\vec{x};z;y_{2,0},y_{2,1};\dots;y_{\ell,0},y_{\ell,1})
  \right)^2.
\end{align*}
The expression inside the square is $\dev_{\ell-1,P}(H_n)$, since the sum is over $\vec{x}$ and $z$
which sums over $k- \ell + 1$ parts of the squashed octahedron with one vertex and a sum over
$\ell-1$ parts of the squashed octahedron with two vertices.  The restriction of $z \in N$
translates to $\mathcal{O}[\vec{x};z;y_{2,0},y_{2,1};\dots;y_{\ell,0},y_{\ell,1}] \subseteq P$,
exactly the restriction in $\dev_{\ell-1,P}(H_n)$.  Thus
\begin{gather*}
  \frac{1}{n^{k+\ell-2}} \left( \dev_{\ell-1,P}(H_n) \right)^2 \leq \dev_{\ell,P}(H_n) =
  o(n^{k+\ell}) \\
  \left( \dev_{\ell-1,P}(H_n) \right)^2 = o(n^{2k + 2\ell - 2}) \\
  \dev_{\ell-1,P}(H_n) = o(n^{k+\ell-1}).
\end{gather*}
\end{proof} 

For completeness, we give the two proofs of Chung~\cite{hqsi-chung90} which were the original
motivation for Lemmas~\ref{lem:subdev} and~\ref{lem:devcauchy}.

\begin{lemma} \label{lem:devtodev}
  \textbf{(Chung~\cite{hqsi-chung90})} For $2 \leq \ell \leq k$, \propdev{$\ell$} $\Rightarrow$
  \propdev{$\ell-1$}.
\end{lemma}

\begin{proof} 
Apply Lemma~\ref{lem:subdev} with $P = V(H)^k$.
\end{proof} 

\begin{lemma} \label{lem:devtocd}
  \textbf{(Chung~\cite{hqsi-chung90})} For $2 \leq \ell \leq k$, \propdev{$\ell$} $\Rightarrow$
  \propcdh{$\ell-1$}.
\end{lemma}

\begin{proof} 
Let $G$ be an $(\ell-1)$-uniform hypergraph.  For $k-\ell+1 \leq i \leq k$, define
\begin{align*}
  P_i = \Bigg\{(x_1,\dots,x_k) \in V(H)^{k} : 
  &\left|\{x_1,\dots,x_{i-1},x_{i+1},\dots,x_k\}\right| = k-1, \\
  &\binom{\{x_1,\dots,x_{i-1},x_{i+1},\dots,x_k\}}{\ell-1} \subseteq E(G) \Bigg\}.
\end{align*}
That is, $P_i$ is the collection of tuples where the vertices besides the $i$th coordinate are
distinct and form a $(k-1)$-clique in $G$.  Note that $P_i$ is complete in coordinate $i$.  We claim
that $\cap P_i$ is the collection of $k$-tuples of distinct vertices which form a $k$-clique in $G$.
Indeed, let $x_1,\dots,x_k$ be distinct vertices forming a $k$-clique of $G$.  Then $(x_1,\dots,x_k)
\in P_i$ for every $i$ since all $(\ell-1)$-subsets of $\{x_1,\dots,x_k\}$ are edges of $G$.  In the
other direction, let $(x_1,\dots,x_k) \in \cap P_i$ and let $R$ be any $(\ell-1)$-subset of
$\{x_1,\dots,x_k\}$.  Since $|R| = \ell - 1$ and $i$ ranges from $k - \ell + 1$ to $k$, there is
some $i$ such that $x_i \notin R$.  But now $(x_1,\dots,x_k) \in P_i$ implies that $R \subseteq
E(G)$ showing that $G[\{x_1,\dots,x_k\}]$ is a clique.  Therefore, Lemma~\ref{lem:subdev} and the
fact that \propdev{$\ell$} holds imply that
\begin{align*}
  \dev_{\ell,\cap P_i}(H_n) \leq \dev_{\ell}(H_n) = o(n^{k+\ell}).
\end{align*}
Now a repeated application of Lemma~\ref{lem:devcauchy} implies that
\begin{align*}
  \dev_{0,\cap P_i}(H_n) = o(n^k).
\end{align*}
Expanding the definition of $\dev_{0,\cap P_i}(H_n)$, we have
\begin{align*}
  \dev_{0,\cap P_i}(H_n) &=
  \sum_{(x_1,\dots,x_k) \in \cap P_i} \eta(x_1;\dots;x_k) \\
  &=k! \sum_{\substack{\{x_1,\dots,x_k\} \subseteq V(H_n) \\
     G[\{x_1,\dots,x_k\}] \,\, \text{is a clique}} }
  \eta(x_1;\dots;x_k) \\
  &= k! \Big( |\mathcal{K}_k(G) \cap E(H)| - \left| \mathcal{K}_k(G) \cap E(\bar H) \right| \Big).
\end{align*}
Thus $\dev_{0,\cap P_i}(H_n) = o(n^k)$ implies that $|\mathcal{K}_k(G) \cap E(H)| = \left|
\mathcal{K}_k(G) \cap E(\bar H) \right| + o(n^k)$ which implies that \propcdh{$\ell-1$} holds for
$\mathcal{H}$.
\end{proof} 

\end{document}